\documentclass[11pt,reqno,tbtags]{amsart}
\usepackage{amssymb}
\usepackage{natbib}
\bibpunct[, ]{[}{]}{;}{n}{,}{,}
\numberwithin{equation}{section}
\allowdisplaybreaks

\newtheorem*{property*}{Property \csname @currentlabel\endcsname}

\newtheorem{theorem}{Theorem}[section]
\newtheorem{lemma}[theorem]{Lemma}

\newtheorem{corollary}[theorem]{Corollary}

\theoremstyle{definition}
\newtheorem{example}[theorem]{Example}
\newtheorem{problem}[theorem]{Problem}
\newtheorem{problems}[theorem]{Problems}
\newtheorem{remark}[theorem]{Remark}

\theoremstyle{remark}

\newenvironment{romenumerate}{\begin{enumerate}
 }{\end{enumerate}}

\newcounter{oldenumi}
{\setcounter{oldenumi}{\value{enumi}}
\begin{romenumerate} \setcounter{enumi}{\value{oldenumi}}}
{\end{romenumerate}}

\newcounter{thmenumerate}
\newenvironment{thmenumerate}
{\setcounter{thmenumerate}{0}%
 \def\item{\par
 \refstepcounter{thmenumerate}\textup{(\roman{thmenumerate})\enspace}}
}
{}

\newcounter{xenumerate}   

\newcommand\pfitem[1]{\par(#1):}

\newcommand{\refT}[1]{Theorem~\ref{#1}}

\newcommand{\refL}[1]{Lemma~\ref{#1}}
\newcommand{\refR}[1]{Remark~\ref{#1}}
\newcommand{\refS}[1]{Section~\ref{#1}}

\newcommand{\refE}[1]{Example~\ref{#1}}

\newcommand{\refand}[2]{\ref{#1} and~\ref{#2}}

\newcommand{\Bigmid}{\,\Big|\,}

\newcommand\marginal[1]{\marginpar{\raggedright\parindent=0pt\tiny #1}}

\begingroup
  \divide\count255 by 60
  \multiply\count255 by -60
  \advance\count255 by \time
  \ifnum \count255 < 10 \xdef\klockan{\the\count1.0\the\count255}
  \else\xdef\klockan{\the\count1.\the\count255}\fi
\endgroup

\newcommand{\sumii}{\sum_{i=1}^\infty}
\newcommand{\sumin}{\sum_{i=1}^n}
\newcommand{\sumj}{\sum_{j=0}^\infty}
\newcommand{\sumji}{\sum_{j=1}^\infty}
\newcommand{\sumk}{\sum_{k=0}^\infty}
\newcommand{\sumki}{\sum_{k=1}^\infty}

\newcommand\set[1]{\ensuremath{\{#1\}}}

\newcommand\Bigset[1]{\ensuremath{\Bigl\{#1\Bigr\}}}

\newcommand\xpar[1]{(#1)}
\newcommand\bigpar[1]{\bigl(#1\bigr)}
\newcommand\Bigpar[1]{\Bigl(#1\Bigr)}
\newcommand\biggpar[1]{\biggl(#1\biggr)}
\newcommand\lrpar[1]{\left(#1\right)}

\def\rompar(#1){\textup(#1\textup)}    
\newcommand\xfrac[2]{#1/#2}

\newcommand\parfrac[2]{\Bigpar{\frac{#1}{#2}}}

\def\xexp(#1){e^{#1}}
\newcommand\ceil[1]{\lceil#1\rceil}
\newcommand\floor[1]{\lfloor#1\rfloor}

\newcommand\ntoo{\ensuremath{{n\to\infty}}}

\newcommand\downto{\searrow}
\newcommand\upto{\nearrow}
\newcommand\half{\tfrac12}
\newcommand\thalf{\tfrac12}

\newcommand\ie{i.e.\spacefactor=1000}
\newcommand\eg{e.g.\spacefactor=1000}

\newcommand\cf{cf.\spacefactor=1000}

\newcommand{\q}{\quad}

\newcommand{\tend}{\longrightarrow}

\newcommand\pto{\overset{\mathrm{p}}{\tend}}

\newcommand\eqd{\overset{\mathrm{d}}{=}}

\newcommand\bbR{\mathbb R}

\newcommand\bbZ{\mathbb Z}

\newcounter{CC} 
\newcommand{\CC}{\stepcounter{CC}\CCx} 
\newcommand{\CCx}{C_{\arabic{CC}}}     
\newcommand{\CCdef}[1]{\xdef#1{\CCx}}     
\newcounter{cc}
\newcommand{\cc}{\stepcounter{cc}\ccx} 
\newcommand{\ccx}{c_{\arabic{cc}}}     
\newcommand{\ccdef}[1]{\xdef#1{\ccx}}     

\newcommand\E{\operatorname{\mathbb E{}}}
\renewcommand\P{\operatorname{\mathbb P{}}}

\newcommand\Cov{\operatorname{Cov}}

\newcommand\Po{\operatorname{Po}}
\newcommand\Bin{\operatorname{Bin}}

\newcommand\gd{\delta}

\newcommand\gl{\lambda}

\newcommand\eps{\varepsilon}

\newcommand\rc{\text{random-cluster}}

\newcommand\cE{\mathcal E}
\newcommand\cG{\mathcal G}
\newcommand\cH{\mathcal H}
\newcommand\cL{{\mathcal L}}
\newcommand\cN{\mathcal N}
\newcommand\cP{\mathcal P}
\newcommand\cS{{\mathcal S}}

\newcommand\cX{{\mathcal X}}

\def\[#1]{[\![#1]\!]}

\newcommand\qq{^{1/2}}
\newcommand\qqw{^{-1/2}}
\newcommand\qw{^{-1}}

\renewcommand{\=}{:=}

\newcommand\oi{[0,1]}
\newcommand\ooox{[0,\infty)}

\newcommand\dtv{d_{\mathrm{TV}}}

\newcommand{\pgf}{probability generating function}

\newcommand\lhs{left hand side}
\newcommand\rhs{right hand side}

\newcommand\gnx[1]{\ensuremath{G_{n,#1}}}
\newcommand\gnxs[1]{G_{n,#1;\cS}}
\newcommand\gmxsi[1]{G_{m,#1;\cS'}}
\newcommand\gnp{\gnx{p}}
\newcommand\gnps{\gnxs{p}}
\newcommand\gnl{\gnx{\gl/n}}
\newcommand\gnln{\gnx{\gl_n/n}}
\newcommand\gnls{\gnxs{\gl/n}}
\newcommand\gnlns{\gnxs{\gl_n/n}}
\newcommand\gmlnsi{\gmxsi{\gl_n/n}}

\newcommand\knxsk[1]{K^{(k)}_{n,#1;\cS}}
\newcommand\knpsk{\knxsk{p}}
\newcommand\knlnsk{\knxsk{\gl_n/n}}

\newcommand\Gnlns{\Gamma_{n,\gl_n/n;\cS}}
\newcommand\Gnps{\Gamma_{n,p;\cS}}
\newcommand\GGnlns{\Gamma^{(2)}_{n,\gl_n/n;\cS}}
\newcommand\GGnps{\Gamma^{(2)}_{n,p;\cS}}
\newcommand\Gnd{\Gamma_{n,\dd}}
\newcommand\GGnd{\Gamma^{(2)}_{n,\dd}}

\newcommand\znxs[1]{Z_{n,#1;\cS}}
\newcommand\znps{\znxs{p}}
\newcommand\znls{\znxs{\gl/n}}
\newcommand\znlns{\znxs{\gl_n/n}}

\newcommand\zxnmus{Z_{n,\nu;\cS}^*}
\newcommand\zxnls{\znls^*}
\newcommand\zxnlns{\znlns^*}
\newcommand\PoS{\Po_{\cS}}
\newcommand\bbZo{\bbZ_{\ge0}}
\newcommand\bbZi{\bbZ_{\ge1}}
\newcommand\nn{\mathbf{n}}
\newcommand\nny{\mathbf{n'}}
\newcommand\ny{n'}
\newcommand\nnz{\mathbf{\widehat n}}
\newcommand\nz{\widehat n}
\newcommand\Nz{\widehat N}
\newcommand\pz{\widehat p}
\newcommand\ppz{\mathbf{\widehat p}}

\newcommand\ppy{\mathbf{\ol p}}
\newcommand\pq{\ol p}
\newcommand\pix{\boldsymbol{\pi}}

\newcommand\cNS{\cN_\cS}
\newcommand\cNSn{\cNS^n}
\newcommand\cnx{\cN^*}
\newcommand\sumjs{\sum_{j\in\cS}}
\newcommand\phis{\phi_\cS}
\newcommand\psis{\psi_\cS}
\newcommand\psisx{\psi_{\cS,1}}
\newcommand\psisxx{\psi_{\cS,2}}
\newcommand\psisxj{\psi_{\cS,j}}
\newcommand\sgraph{$\cS$-graph}
\newcommand\ssubgraph{$\cS$-subgraph}
\newcommand\smultigraph{$\cS$-multigraph}
\newcommand\egl{\ensuremath{E(\gl)}}
\newcommand\eglo{\ensuremath{E_0(\gl)}}
\newcommand\eglx[1]{\ensuremath{E(#1)}}
\newcommand\eglox[1]{\ensuremath{E_0(#1)}}
\newcommand\mux{{\widehat\mu}}
\newcommand\muxmin{{\widehat\mu}_*}
\newcommand\muxmax{{\widehat\mu}^*}
\newcommand\mui{\mu'}

\newcommand\gammax{\widehat{\gamma}}
\newcommand\zetax{\widehat{\zeta}}
\newcommand\xix{\widehat{\xi}}
\newcommand\chimux{\chi_{\mux}}
\newcommand\rx{\widehat r}

\newcommand\xmu{X_\mu}
\newcommand\xmux{X_\mux}
\newcommand\xmus{X_{\mu,\cS}}
\newcommand\wmu{W_{\mu}}
\newcommand\wmur{W_{\mu,r}}
\newcommand\gnmux{\gnx\nu^*}
\newcommand\gnmuxs{\gnxs\nu^*}
\newcommand\gxnln{\gnln^*}
\newcommand\gxnlns{\gnlns^*}
\newcommand\gxmn{\gmlnsi^*}
\newcommand\gxnl{\gnl^*}
\newcommand\gxnls{\gnls^*}
\newcommand\gxn{\cG^*_n}
\newcommand\gxns{\cG^*_{n;\cS}}
\newcommand\gnd{\gnx{\dd}}
\newcommand\gnddx{\gnx{\dd}^*}
\newcommand\prodin{\prod_{i=1}^n}
\newcommand\ann{z(\nn)}
\newcommand\anny{z(\nny)}
\newcommand\annz{z(\nnz)}
\newcommand\ooo{_0^\infty}
\newcommand\sumks{\sum_{k\in\cS}}

\newcommand\wwi{^{(1)}}

\newcommand\wwx[1]{^{(#1)}}
\newcommand\ee{\mathbf e}
\newcommand\aae{A_2^\eps}
\newcommand\aao{A_2^0}
\renewcommand\O{\text{\rm O}}
\renewcommand\o{\text{\rm o}}
\newcommand\edg{\text{\rm e}}
\newcommand\ver{\text{\rm v}}

\newcommand\ol[1]{\overline{#1}}
\newcommand\dn{\ensuremath{(d_i)_1^n}}
\newcommand\dd{\ensuremath{\mathbf d}}

\newcommand\simple{\text{\rm\ is simple}}
\newcommand\glo{\gl_0}
\newcommand\muo{\mu_0}
\newcommand\mua{\mu_1}
\newcommand\ddmu{\frac{d}{d\mu}}
\newcommand\dddmu{\ddmu}
\newcommand\ddpmu{\frac{\partial}{\partial\mu}}
\newcommand\ddpgl{\frac{\partial}{\partial\gl}}
\newcommand\Lex{\widetilde\Lambda}
\newcommand\glx{\hat\gl}
\newcommand\ioi{\bigpar{1+\o(1)}}
\newcommand\lest{\le_{\text{\rm st}}}

\newcommand\SI{{\cS-1}}
\newcommand\PoSI{\Po_{\SI}}
\newcommand\cXX{\ol{\mathcal X}}

\newcommand\CS{Cauchy--Schwarz}
\newcommand\CSineq{\CS{} inequality}

\newcommand{\maple}{\texttt{Maple}}

\newcommand\REM[1]{{\raggedright\texttt{[#1]}\par\marginal{XXX}}}

\newcommand\urladdrx[1]{{\urladdr{\def~{{\tiny$\sim$}}#1}}}

\begin{document}
\title[Random graphs with forbidden vertex degrees]
{Random graphs with\\ forbidden vertex degrees}

\date{December 1, 2007 (typeset \today{} \klockan)} 

\author{Geoffrey Grimmett}
\address{Statistical Laboratory, Centre for Mathematical Sciences,
Cambridge University, Wilberforce Road, Cambridge CB3 0WB, UK}
\email{g.r.grimmett@statslab.cam.ac.uk}
\urladdrx{http://www.statslab.cam.ac.uk/~grg/}

\author{Svante Janson}
\address{Department of Mathematics, Uppsala University, PO Box 480,
SE-751~06 Uppsala, Sweden}
\email{svante.janson@math.uu.se}
\urladdrx{http://www.math.uu.se/~svante/}


\begin{abstract} 
We study the random graph $G_{n,\gl/n}$ conditioned on the event
that all vertex degrees lie in some given subset $\cS$ of the non-negative
integers. Subject to a certain hypothesis on $\cS$, the 
empirical distribution of the vertex degrees is
asymptotically Poisson
with some parameter $\mux$ given as the root of a certain
`characteristic equation' of  
$\cS$ that maximises a certain function $\psis(\mu)$. 
Subject to a hypothesis on $\cS$,
we obtain a partial description
of the structure of such a random graph, including a condition
for the existence (or not) of a giant component. The requisite
hypothesis
is in many cases benign, and applications are presented to a
number of choices for the set $\cS$ including the sets of (respectively)
even and odd numbers. The random \emph{even} graph is related to
the \rc\ model on the complete graph $K_n$.
\end{abstract}

 
\maketitle

\section{Introduction}\label{Sintro}

Let $\cS$ be a fixed nonempty set of non-negative integers.
The purpose of this paper is to study the structure of random graphs
having all their vertex degrees restricted to the set $\cS$.

We call a graph an \emph{\sgraph} if all its
vertex degrees belong to $\cS$. For example, if $\cS=\set s$ is a
singleton, an \sgraph{} is the same as a regular graph of degree $s$. 
(We are not going to say anything new about this case.)
One of our main examples is the class of \emph{Eulerian graphs}, or
\emph{even graphs}, given by the set of even numbers
$\cS=2\bbZo$, with $\bbZo$ the set $\{0,1,2,\dots\}$ of non-negative integers.
See \refS{Sex} for further examples.

More precisely, 
we will study the random graph $\gnps$ defined as $\gnp$
\emph{conditioned on being an \sgraph},
where $\gnp$ is the standard 
random subgraph of the (labelled) complete graph $K_n$
where two vertices are joined by an edge with probability $p\in(0,1)$, and
these $\binom n2$ events, corresponding to the edges of $K_n$, are
independent.
In other words, $\gnps$ is a random $\cS$-subgraph of $K_n$ such that, if $G$
is any given subgraph of $K_n$ that is an \sgraph, then
\begin{equation}\label{gnps}
  \P(\gnps=G)
=
\frac{p^{\edg(G)}(1-p)^{\binom n2-\edg(G)}}{\P(\gnp \text{ is an \sgraph})},
\end{equation}
where  $\edg(G)$ is the number of edges of $G$.
We are interested in asymptotics as \ntoo, and we will tacitly 
consider only $n$ such that there exists an $\cS$-graph with  $n$ vertices, in
other words, such that the denominator in \eqref{gnps} is non-zero.
Thus, 
a finite number of small $n$ may be excluded;
moreover, if $\cS$ contains only odd integers, then $n$ has to be
even.
(It is easy to see that, apart from this parity restriction, all large
$n$ are allowed.)

\begin{remark}
The choice $p=\frac12$ gives a random \sgraph{} that is uniformly
distributed over all \sgraph{s}
on $n$ labelled vertices.
However, we will in this paper instead study the case when $p$ is of order
$1/n$ and the average vertex degree  is bounded
(for $\gnp$, and as we shall see
later, for $\gnps$ too).
\end{remark}

It follows immediately from \eqref{gnps} that two \ssubgraph{s} of
$K_n$ with the same degree sequence are attained with the same
probability. Hence, the conditional distribution of $\gnps$ given the
degree sequence is uniform. We will therefore focus on studying the
random degree sequence of $\gnps$; it is then possible to 
obtain further results on the structure of $\gnps$ by
applying
standard results on random graphs with given degree sequences
to $\gnps$ conditioned on the degree sequence.
For example, using the results by \citet{MR95,MR98} we obtain
\refT{Tgiant} below on existence of a giant component in $\gnps$. 

\section{Main theorem}\label{S1.5}

By symmetry, the labelling of the vertices and thus the order of the
degree sequence is not important, and we shall therefore 
study the numbers of vertices with given
degrees, rather than
the degree sequence itself.
We introduce some notation.

Let $\cN$ be the set of sequences $\nn=(n_0,n_1,\dots)$ of non-negative
integers $n_j\in\bbZo$
with only a finite number of non-zero terms $n_j$.
Let 
\begin{equation*}
 \cNS\=\set{\nn\in\cN:n_j=0 \text{ when } j\notin\cS} ,
\end{equation*}
the set of such sequences supported on $\cS$.
For a (multi)graph $G$, let $n_j(G)$ be the number of vertices of
degree $j$ in $G$, $j\in\bbZo$, and let
$\nn(G)\=(n_j(G))_{j=0}^\infty$ be the sequence of degree counts.
Thus, $G$ is an \sgraph{} if and only if $\nn(G)\in\cNS$.
Clearly, \cf{} \eqref{gnps},
\begin{equation}
  \P(\gnp \text{ is an \sgraph})
= \sum_{G:\,\nn(G)\in\cNS} 
p^{\edg(G)}(1-p)^{\binom n2-\edg(G)}.
\label{partfn}
\end{equation}
We shall call this summation the \emph{partition function}, 
denoted as $\znps$.

Let $\cP$ be the set of probability distributions on $\bbZo$. In other words,
$\cP$ is the set of
sequences $\pix=(\pi_0,\pi_1,\dots)$ of non-negative real numbers
such that $\sum_j\pi_j=1$.
We regard $\cP$ as a topological space with the usual topology of weak
convergence (denoted $\pto$); 
it is well known that this topology on $\cP$ may be metrised by the total
variation distance 
\begin{equation*}
\dtv(\pix,\pix')\=\half\sum_j|\pi_j-\pi'_j|.  
\end{equation*}
If $G$ has $n=\sum_jn_j(G)$ vertices, let
$\pi_j(G)\=n_j(G)/n$, the proportion of vertices of degree $j$, and
$\pix(G)\=\nn(G)/n=(\pi_j(G))_{j=0}^\infty$.
Note that $\pix(G)$ is the probability distribution of the degree
of a randomly chosen vertex in $G$.

Let $\phis$ be the exponential generating function of $\cS$,
\begin{equation}\label{phis}
  \phis(\mu)\=\sumjs \frac{\mu^ j}{j!}.
\end{equation}
Note that this is an entire function of $\mu$, and that $\phis(\mu)>0$
for $\mu>0$ while $\phis(0)>0$ if and only if $0\in\cS$.

Let $\PoS(\mu)$ be the distribution of a $\Po(\mu)$ distributed
variable given that it belongs to $\cS$, \ie, $\cL(\xmu\mid \xmu\in\cS)$
with $\xmu\sim\Po(\mu)$. Thus, recalling \eqref{phis},
\begin{equation}\label{pos}
  \PoS(\mu)\set{k}=\frac{\mu^k/k!}{\phis(\mu)},
\qquad k\in\cS.
\end{equation}
This conditional distribution is always defined for $\mu>0$, and in
the case $0\in\cS$ for $\mu=0$ too (in which case it is
a point mass at 0).
The mean of the $\PoS(\mu)$ distribution is
\begin{equation}
  \label{epos}
\E(\xmu\mid \xmu\in\cS)
=
\frac1{\phis(\mu)}
\sumks \frac{k\mu^k}{k!} 
=\frac{\mu\phis'(\mu)}{\phis(\mu)}.
\end{equation}

Let $\gl > 0$. We shall refer to the equation
\begin{equation}\label{chareq}
\frac{\mu\phis'(\mu)}{\phis(\mu)}=\frac{\mu^ 2}{\gl},
\end{equation}
as the \emph{characteristic equation} of the set $\cS$
(for this value of $\gl$),
and we write
\begin{equation}\label{egl}
\egl
\=
\Bigset{ \mu \ge 0: \frac{\mu\phis'(\mu)}{\phis(\mu)} = \frac{\mu^ 2}{\gl} },
\end{equation}
where we allow $\mu=0$ only if\/ $0\in\cS$.
We further define the auxiliary function 
\begin{equation}\label{psidef}
  \psis(\mu)
\=
\log\phis(\mu)-\frac{\mu\phis'(\mu)}{2\phis(\mu)}
\end{equation}
and note that when $\mu\in\egl$, $\psis(\mu)$ equals the simpler function
\begin{equation}\label{psixdef}
  \psisx(\mu;\gl)
\=
\log\phis(\mu)-\frac{\mu^2}{2\gl}.
\end{equation}

All logarithms in this paper are natural.
We let 
$c_1,C_1,\dots$ denote positive constants, generally depending on
$\cS$ 
and
$\gl$
(or $(\gl_n)$)
and sometimes on other parameters too (but not on $n$),
which may be indicated by arguments. We sometimes assume that $n>1$ to
avoid trivialities.

\begin{theorem}\label{T1}
Let $\gl_n\to\gl>0$ and suppose that $\egl$ contains a unique $\mux=\mux(\gl)$
that maximizes 
$\psis(\mu)$ (or, equivalently, $\psisx(\mu;\gl)$)
over \egl.
Then, the following hold, as \ntoo: 
\begin{romenumerate}
\item\label{T1a}
$\pix(\gnlns)\pto\PoS(\mux)$.
In other words, for every $j\in\cS$,
\begin{equation}\label{t1a}
\frac{n_j(\gnlns)}{n}\pto\PoS(\mux)\set{j}
=\frac{\mux^j/j!}{\phis(\mux)}.
\end{equation}

\item\label{T1b}
All moments of the random distribution 
$\pix(\gnlns)$ converge to the corresponding moments of $\PoS(\mux)$.
In other words, if $d_1,\dots,d_n$ is the degree sequence of $\gnlns$,
and $\xmu\sim\Po(\mu)$, then for every $r\in(0,\infty)$, 
\begin{equation}\label{t1b}
\begin{aligned}
\sumin \frac{d_i^r}{n}&=
\sumj \frac{j^r n_j(\gnlns)}{n}
\\
&\pto \sumj j^r\PoS(\mux)\set{j}
=\E(\xmux^r\mid \xmux\in\cS).
\end{aligned}
\end{equation}

In particular, 
\begin{equation}
\frac{e(\gnlns)}{n}
\pto \half\E(\xmux\mid \xmux\in\cS)
=\frac{\mux\phis'(\mux)}{2\phis(\mux)}.
\label{t1bb}
\end{equation}

\item\label{T1c}
The error probabilities in \ref{T1a} decay exponentially:  for every
$\eps>0$, there exists a constant
$\cc=\ccx(\eps,\gl,\cS)>0\ccdef\cctic$ 
such that, for all
large $n$,
\begin{equation}\label{t1c}
  \P\lrpar{\dtv\bigpar{\pix(\gnlns),\PoS(\mux)}\ge\eps}
\le e^{-\ccx n}.\ccdef\cctic
\end{equation}

\item\label{T1d} We have that
\begin{equation}\label{t1d}
\frac 1n\log\znlns=
 \frac1n \log \P(\gnln \text{\rm\ is an \sgraph})
\to \psi_\cS(\mux)-\tfrac12\gl.
\end{equation}
\end{romenumerate}
\end{theorem}

More generally, let \eglo{} be the subset of \egl{} where $\psis(\mu)$
(or $\psisx(\mu;\gl)$) is maximal:
\begin{equation}
  \eglo\=\left\{\mu\in\egl:\psis(\mu)=\max_{\mui\in\egl}\psis(\mui)\right\}.
\end{equation}
If \eglo{} contains a single element, we thus take that element as
$\mux(\gl)$; in particular, if $|\egl|=1$, then
$\egl=\eglo=\set{\mux(\gl)}$.
We shall see in \refS{Smu}
that this is the normal case: $\egl$ is
always finite and non-empty, and $|\eglo|=1$ except for at most a
countable number of values of $\gl$. 

Further results on $\mux$ and the auxiliary functions are given in \refS{Smu}.

\begin{remark}\label{Rphase}
The set $\eglo$ may contain more than one element, see
\refE{E03}. 
In this case, the theorem
may not be applied, but 
the proof in Sections \ref{SpfT1x}--\ref{SpfT1} extends to show that
the random degree distribution $\pi(\gnlns)$ approaches the finite set
$F:=\set{\PoS(\mu):\mu\in\eglo}$ in the sense that the analogue of
\eqref{t1c} holds for the distance to this set, \ie, 
\begin{equation}
  \P\Bigpar{\min_{\mu\in\eglo}\dtv\bigpar{\pix(\gnlns),\PoS(\mu)}\ge\eps}
\le e^{-\cctic n}.
\end{equation}

We can regard the distributions in $F$ as pure phases, in analogy with
the situation for many infinite systems of interest in statistical
physics, but for the finite systems considered here this has to be
interpreted asymptotically. Thus, for large $n$, the degree
distribution of $\gnlns$ is approximately given by one of the pure
phases, but we do not know which one. It follows that if we let
\ntoo, one of the following happens for the random degree distribution
$\pi=\pi(\gnlns)$ (regarded as an element of $\cP$): 
\begin{romenumerate}
  \item
$\pi$ converges in probability to $\PoS(\mu)$ for
some $\mu\in\eglo$.
  \item
$\pi$ converges in distribution to some non-degenerate distribution on
$F$. (A mixture of two or more pure phases.)
\item
There are oscillations and $\pi$ does not converge in distribution;
suitable subsequences converge as in (i) or (ii), but different
subsequences may have different limits.
\end{romenumerate}
It is easy to show by a continuity argument that all three cases
may occur in \refE{E03} for suitable sequences $(\gl_n)$ (with
$\gl_n\to\glo$ defined there). We do not know whether all three cases
may occur for fixed $\gl$.

We shall not investigate the case $|\eglo|>1$ further here.
\end{remark}

\begin{remark}
The reason for taking $n$ large in \eqref{t1c} is as follows.
Suppose, for example, that $\cS$ is
an infinite set.
Then the \lhs{} of \eqref{t1c} trivially equals 1 for any fixed
$n$ and sufficiently small $\eps$. One way around this, at least when
$\gl_n\equiv\gl$,
is to replace the \rhs{} of \eqref{t1c} by $2e^{-\cc n}$ for suitable $c_2>0$.
\end{remark}

We close this section with an informal 
explanation of the results of Theorem \ref{T1}, 
several applications of which are presented in Section \ref{Sex}.
Recall the partition function $\znls$ of \eqref{partfn}, considered
as a summation over suitable graphs.
We wish to establish which graphs are dominant in this summation. 
In so doing, we will treat certain discrete variables as
continuous, and shall study maxima by differentiation
and Lagrange multipliers.
Let
$\pix=(\pi_0,\pi_1,\dots)$ be a sequence of non-negative reals
satisfying $\sum_i\pi_i=1$,
and $\pi_i = 0$ for $i\notin \cS$. We write 
$$
\nu=\nu(\pix)=\sum_i i\pi_i.
$$
Let $Z(\pix)$ represent the contribution to the summation of
\eqref{partfn} from graphs $G$ having, for each $i$, approximately $n\pi_i$
vertices with degree $i$.  The (empirical) mean vertex-degree of such a graph
is $\nu$. 

Now, $Z(\pix)$ is a summation over simple graphs subject to
constraints on the vertex degrees. It may be approximated by a similar
summation $Z'(\pix)$ over certain multigraphs, 
and this is easier to express in closed form,
as follows. The number of ways of partitioning $n$ vertices into sets
$V_0,V_1,\dots$ of respective sizes $n\pi_i$, $i\ge 0$, is 
$$
\frac {n!}{(n\pi_0)!(n\pi_1)!\cdots}.
$$
Each vertex $v\in V_i$ will be taken to have degree $i$, and we therefore provide
$v$ with $i$ `half-edges'.  Each such half-edge will be connected to
some other half-edge to make a whole edge. 
Since half-edges are considered indistinguishable, 
we shall require the multiplicative factor
$$
\left\{\prod_{i\in \cS} (i!)^{-n\pi_i}\right\}.
$$
The total number of half-edges is 
$2N = n \sum_i i\pi_i =n\nu$, and we assume for simplicity
that $N$ is an integer. These half-edges may be
paired together in any of
$$
(2N-1)!! = (2N-1)(2N-3) \cdots 3\cdot 1 = \frac{(2N)!}{2^{N}N!}
$$
ways, and each such pairing contributes 
$$
\left(\frac\gl n\right)^{N} \left(1-\frac \gl n\right)^{\binom n2 - N}
$$
to $Z'(\pix)$.
We combine the above to obtain an approximation to $Z(\pix)$:
$$
Z(\pix) \approx \frac {n!}{(n\pi_0)!(n\pi_1)!\cdots}
\left\{\prod_{i\in \cS} (i!)^{-n\pi_i}\right\}
\frac{(2N)!}{2^{N}(N)!} \left(\frac\gl n\right)^{N} 
\left(1-\frac \gl n\right)^{\binom n2 - N}.
$$

By Stirling's formula, as $n\to\infty$,
\begin{equation}
\frac 1n \log Z(\pix) \to \tfrac12 \nu \log(\gl\nu/e)- \tfrac12 \gl
-\sum_{i\in\cS} \pi_i \log(i!\,\pi_i).
\label{expasympt}
\end{equation}
We maximize the last expression subject to
$\sum_i \pi_i = 1$ to find that
\begin{equation}
\pi_i = A \frac{\mu^i}{i!},\qquad i\in\cS,
\label{pimu}
\end{equation}
for some constant $A$ and some $\mu$ satisfying 
\begin{equation}
\mu=\sqrt{\gl \nu}.
\label{chareq2}
\end{equation}
Thus $\pix$ is the mass function of the
$\PoS(\mu)$ distribution and, by \eqref{epos} and the definition of $\nu$,
\begin{equation}
\nu = \frac{\mu\phis'(\mu)}{\phis(\mu)}.
\label{constr2}
\end{equation}
We combine this with \eqref{chareq2} to obtain the `characteristic equation' 
\eqref{chareq}.

If there exists a unique $\mu$ satisfying the characteristic equation,
then we are done. If there is more than one, we pick the value that
maximizes the right hand side of \eqref{expasympt}.
That is to say, the exponential asymptotics of $\znls$ are dominated by
the contributions from graphs with $\pix$ satisfying \eqref{pimu}
with $\mu$ chosen to satisfy the characteristic equation and
to maximize $\psi_\cS(\mu)$.

Note from \eqref{expasympt} that
\begin{equation}
\frac 1n \log Z(\pix) \to \psi_\cS(\mu) - \tfrac12 \gl,
\label{expasympt2}
\end{equation}
and part (iv) of Theorem \eqref{T1} is explained.

We make the above argument rigorous in the forthcoming proof of 
Sections \ref{SpfT1x}--\ref{SpfT1}, a substantial
part of which is devoted to proving that the conditional
distribution of vertex degrees is concentrated near its mode.

\section{The giant and the core}\label{giantcore}

We show next how to apply \refT{T1}, in conjunction
with results of \citet{MR95,MR98} and \citet{JL,SJ204}, to
identify the sizes of the giant cluster and the $k$-core
of $\gnlns$. The proofs are deferred to Section \ref{sec4.5}.

We consider first the existence or not of a giant component
in the random $\cS$-graph $\gnlns$ as $\gl_n\to\gl > 0$. Let
$\pix=(\pi_0,\pi_1,\dots)$ be a vector of non-negative reals with sum 1,
and write $\nu=\sum_j j\pi_j$.
As explained
in \cite{MR95,MR98}, 
if we consider the random graph with 
given degree sequence $\dd=\dn$, and assume that 
there are $n(1+\o(1))\pi_j$
vertices with degree $j$, the quantity that is key to
the existence of a giant component is 
$$
Q(\pix)\=\sum_j j(j-2)\pi_j.
$$
Subject to certain conditions, if $Q(\pix)>0$,
there exists a giant component, while
there is no giant component when $Q(\pix)\le0$. 

We shall apply this
with $\pi_j=\PoS(\mux){\set{j}}$, and to that end we introduce some
further notation. Let, see \eqref{epos},
\begin{align}
\nu(\mu)&\= \sum_j j\PoS(\mu){\set{j}}
=\frac{\mu\phis'(\mu)}{\phis(\mu)},
\label{nudef}
\\
Q(\mu) &\= \sum_j j(j-2)\PoS(\mu){\set{j}} =
\frac{\mu^2\phis''(\mu)-\mu\phis'(\mu)}{\phis(\mu)},
\label{Qdef}
\\  
\chi_\mu(\xi) &\= 
\sum_j j\PoS(\mu){\set{j}}(\xi^2-\xi^j)
=
\nu(\mu)\xi^2  - \mu\xi \frac{\phis'\bigl(\mu\xi\bigr)}{\phis(\mu)}.
\label{chidef}
\end{align}
Note that $\chi_\mu(0)=\chi_\mu(1)=0$.
Furthermore, the only possibly negative term in the sums in \eqref{Qdef}
and \eqref{chidef} for $\xi\in(0,1)$ are
those with $j=1$, while the terms with $j=0$ and $j=2$ always
vanish and the others are positive unless $\PoS(\mu)=0$.

Let $\Gnps$ be 
the component
of $\gnps$ with the largest number of vertices, and let
$\GGnps$ be the second largest. (Break ties by any rule.)
We write $\ver(H)$ for the number of vertices 
in a graph $H$.

\begin{theorem}
  \label{Tgiant}
Suppose that $\cS\not\subseteq\set{0,2}$.
Let $\gl_n\to\gl>0$ and suppose that $\eglo$ contains a unique element
$\mux\ge0$. 
Then, $\gnlns$ has a giant component if and only if $Q(\mux)>0$, \ie,
if and only if $\mux\phis''(\mux)>\phis'(\mux)$.
More precisely,
as $n\to\infty$,
\begin{align*}
  n^{-1}\ver(\Gnlns) &\pto \gammax\ge0, 
 &
  n^{-1}\edg(\Gnlns) &\pto \zetax\ge0,
\\
  n^{-1}\ver(\GGnlns) &\pto 0, 
 &
  n^{-1}\edg(\GGnlns) &\pto 0,
\end{align*}
where
\begin{align}
  \gammax&=\gammax(\mux)
\=1-\sum_j \xix^j\PoS(\mux)\set{j}
=1-\frac{\phis(\xix\mux)}{\phis(\mux)},
\label{gammax}
\\
  \zetax&=\zetax(\mux)
\=\tfrac12 \nu(\mux) \bigpar{1-\xix^2}
\label{zetax}
\end{align}
with $\xix=\xix(\mux)\in\oi$ given as follows:
\begin{romenumerate}
\item 
if $Q(\mux) >0$ and $1\in\cS$, then 
$\xix\in(0,1)$ is the unique solution to $\chimux(\xix)=0$ with
$0<\xix<1$, and $\gammax,\zetax>0$;
\item
if $Q(\mux) >0$ and $1\notin\cS$, then 
$\xix=0$ and $\gammax=1-1/\phis(\mux)>0$, $\zetax=\frac12 \nu(\mux)>0$;
\item 
if $Q(\mux)\le0$, then $\xix=1$ and $\gammax=\zetax=0$.
\end{romenumerate}
\end{theorem}

\begin{remark}\label{Rgiant0}
If $\mux=0$, then $Q(\mux)=0$ and we are in Case (iii) with no giant
component; in fact, by \refT{T1}, $n_0/n\pto1$ so almost all vertices
are isolated. In this case $\chimux(\xi)=0$ for all $\xi$.

If $\mux>0$, then $\chimux<0$ on $(0,\xix)$  and $\chimux>0$ on
$(\xix,1)$ in all three cases, 
as follows from \cite[Lemma 5.5]{SJ204}, 
which yields another characterization of $\xix$.
\end{remark}

\begin{remark}
  If $1\notin\cS$, then $Q(\mux)>0$ as soon as $\mux>0$. Hence we are
  in Case (ii) if $\mux>0$ and in Case (iii) in $\mux=0$.
\end{remark}

\begin{remark}
  We have excluded the cases $\cS\subseteq\set{0,2}$, \ie, the trivial
  case $\cS=\set0$ and the cases $\set2$ and \set{0,2} that are
  exceptional; in the latter cases 
$n^{-1}\ver(\Gnlns)$ has a continuous limiting distribution and thus
not a constant limit when $\mux>0$, see Examples \refand{Es}{E02}.
We note also that $Q(\mu)=0$ for all $\mu$ in the excluded cases.
\end{remark}

\begin{remark}
  \label{RGW}
It is easily seen, using \eqref{chidef}, that $\xix$ equals the
extinction probability of a Galton--Watson process with offspring
distribution
\begin{align*}   
\P(X=j-1)
&=
\frac{j\PoS(\mux)\set j}{\nu}\\
&=\frac{\mux^j}{(j-1)!\,\phis(\mux)}\parfrac{\mux\phis'(\mux)}{\phis(\mux)}\qw,
\qquad j\in \cS\setminus\{0\},
\end{align*}  
that is, the distribution $\PoSI(\mux)$ where
$\SI\=\set{k\ge0:k+1\in\cS}$. (Note that $\phi_\SI(\mu)=\phis'(\mu)$.)
Hence $\gammax$, the asymptotic relative size of $\Gnlns$, equals by
\eqref{gammax} the survival probability of a Galton--Watson process
with offspring distribution $\PoSI(\mux)$ and initial distribution
$\PoS(\mux)$. 
\end{remark}

The $k$-core of a graph $G$ is the largest induced
subgraph having mimimum vertex degree at least $k$.
The $k$-core of an Erd\H os--Re\'nyi random
graph has attracted much attention; see \cite{JL} and the references therein.
\refT{T1} may be applied in conjunction with Theorem 2.4 of \citet{JL}
to obtain the asymptotics of the $k$-core of $\gnlns$. 
Let $\knpsk$ denote the $k$-core of $\gnps$. We shall require some further
notation in order to state our results for $\knlnsk$.

Let $k\in\{2,3,\dots\}$.
Let $\mu\ge 0$, and let $\wmu$ be a random variable with the $\PoS(\mu)$
distribution.  
For $r\in[0,1]$,
let $\wmur$ be obtained by `thinning' $\wmu$ at rate $1-r$ so that,
conditional on $\wmu$, 
$\wmur$ has the binomial distribution $\Bin(\wmu,r)$. For
$k\in\{2,3,\dots\}$, let 
\begin{align*}
h_{\mu,k}(r) 
&= \E(\wmur I_{\{\wmur\ge k\}}) 
= \sum_{l=k}^\infty l \P(\wmur=l),\\
\ol h_{\mu,k}(r) &= \P(\wmur \ge k).
\end{align*}

\begin{theorem}
  \label{Tcore}
Let $\gl_n\to\gl>0$ and suppose that $\eglo$ contains a unique element
$\mux$. Let $k \ge 2$, and  
let, with $\nu=\nu(\mux)$ as above,
$$
\rx=\sup\{r\le 1: \nu r^2 = h_{\mux,k}(r)\}.
$$
As $n\to\infty$:
\begin{romenumerate}
\item
if $\rx=0$, 
\begin{equation*}
\frac 1n \ver(\knlnsk) \pto 0,\qquad
\frac 1n \edg(\knlnsk) \pto 0;
\end{equation*}
if, further, $k \ge 3$, then 
$$
\P(\knlnsk \text{\rm\ is empty})\to 1;
$$

\item
if $\rx>0$, and in addition $\nu r^2 < h_{\mux,k}(r)$ on some
non-empty interval 
$(\rx-\eps,\rx)$, then 
\begin{align*}
\frac 1n \ver(\knlnsk) &\pto \ol h_{\mux,k}(\rx),\\
\frac 1n \edg(\knlnsk) &\pto \frac12 h_{\mux,k}(\rx)= \frac12 \nu \rx^2.
\end{align*}
\end{romenumerate}
\end{theorem}

\begin{remark}\label{RcoreGW}
  Let $\cX_\mu$ be the Galton--Watson process with offspring
  distribution $\PoSI(\mu)$, started with a single individual $o$,
and let $\cXX_\mu$ be the modified process where the first
generation has distribution $\PoS(\mu)$, \cf{} \refR{RGW}.
It may be seen that
$\rx$ is
the probability that the family tree of $\cX_\mu$ contains an
infinite subtree with root $o$ and every node having $k-1$ children.
Similarly,  $\ol h_{\mux,k}(\rx)$ equals the probability
that $\cXX_\mu$ contains an infinite $k$-regular subtree
with root $o$ (the root has $k$ children and all other vertices have $k-1$).
It is easy to see heuristically that this yields the asymptotic
probability that a random vertex belongs to the $k$-core, 
see \citet{psw96} (for $\cS=\bbZo$), 
but it is
difficult to make a proof based on branching process theory; see
\citet{Riordan:kcore} where this is done rigorously 
for another random graph model.
\end{remark}

\section{Roots of the characteristic equation}\label{Smu}

To avoid some trivial complications, we assume throughout this section
that $\cS\neq\set0$, thus excluding the trivial case $\cS=\set0$ for which
$\gnlns$ comprises isolated vertices only.

\begin{lemma}\label{Lmu}
$\phis'(\mu)\le\CC\phis(\mu)$ for all $\mu\ge1$.
\CCdef\CClmu
\end{lemma}

\begin{proof}
  By \eqref{phis},
\begin{equation*}
  \phis'(\mu)=\sumjs j\frac{\mu^{j-1}}{j!}.
\end{equation*}
We split the sum into two parts. For $j\le4\mu$, 
\begin{equation*}
\sum_{j\in\cS,\;j\le4\mu} \frac{j\mu^{j-1}}{j!}
\le
\sumjs \frac{4\mu^{j}}{j!}
=4\phis(\mu).
\end{equation*}
For $j>4\mu$, Stirling's formula implies
\begin{equation*}
 \frac{j\mu^{j-1}}{j!}
\le j\mu^{j-1}\parfrac{e}{j}^j
\le 4^{1-j}e^j
\end{equation*}
and thus, for $\mu\ge1$,
\begin{equation*}
\sum_{j\in\cS,\;j>4\mu} \frac{j\mu^{j-1}}{j!}
\le
\sumji 4 \parfrac e4 ^j
=\CC
=\CC\phis(1)
\le\CCx\phis(\mu).
\qedhere
\end{equation*}
\end{proof}

\begin{theorem}
  \label{Tegl}
For each $\gl>0$, the set $\egl$ is finite and non-empty.
\end{theorem}

\begin{proof}
The characteristic equation \eqref{chareq}
may be written as $h(\mu)=0$ where 
$h(\mu) = \gl\mu\phis'(\mu)-\mu^2\phis(\mu)$. 
By \refL{Lmu}, for $\mu\ge1$,
\begin{equation}
h(\mu)\le(\CClmu\gl\mu-\mu^2)\phis(\mu), 
\label{hbnd}
\end{equation}
and thus $h(\mu)<0$ for
$\mu>\CC=\max\{\CClmu\gl,1\}$.

Since $h$ is an entire function, and does not vanish identically by
what we just have shown,
it has only finitely many zeros in
each bounded subset of the complex plane, and in particular in the
interval $[0,\CCx]$. Hence $\egl$ is finite.

To see that $\egl$ is non-empty,
let $s$ be the smallest element of $\cS$. If $s=0$, then $0\in\egl$.
If $s>0$, then $h(\mu)\sim\gl s \mu^s/s!$
as $\mu\to0$, so $h(\mu)>0$ for
small positive $\mu$. Since further $h(\mu)$ is negative for large $\mu$,
$h$ possesses a zero on the positive real axis.
\end{proof}

We have defined $\mux$ as the maximum point of $\psis$ or $\psisx$ 
on $\egl$. 
The next theorem shows that, alternatively, 
it can be defined as the maximum point of
$\psisx$ on $[0,\infty)$ (but not of $\psis$).
Furthermore, instead of $\psisx$, we can use the function 
\begin{equation}\label{psisxx}
\psisxx(\mu;\gl)
\=
\log\phis (\mu)
+ \frac{\mu\phis'(\mu)}{2\phis(\mu)}
\Bigpar{\log \frac{\gl\phis'(\mu)}{\mu\phis(\mu)}-1},
\end{equation}
that arises as follows. In \refS{Smulti},
we will indicate the use of multigraphs in proving \refT{T1},
of which we shall derive a multigraph equivalent at \refT{T1x}.
With $\zxnls$ denoting the multigraph partition function,
we shall see in the proof of \refT{T1x} 
that $\psisxx(\mu;\gl)$ represents the contribution to
$n\qw\log\zxnls$ from multigraphs with degree distribution close to
$\PoS(\mu)$,
see \refR{Rpsi2}.
For this reason, $\psisxx$ is a more natural function than $\psisx$,
although it 
has a more complicated formula.
We shall have to exclude the trivial case when $\cS=\set s$ is a
singleton; in this case
$\psisxx(\mu;\gl)= \frac 12 s\bigpar{\log(\gl s)-1}-\log(s!)$ is constant.

It is easily seen that $\psisxx(\mu;\gl)\ge\psisx(\mu,\gl)$ for all
$\mu$ and $\gl>0$, with equality if and only if $\mu\in\egl$.

We regard $\psisx$ and $\psisxx$ 
as functions of $\mu$, with $\gl$ considered a fixed parameter.
These functions are evidently analytic on $(0,\infty)$.
Note that if $0\in\cS$, then $\phis(0)=1$ and $\psis(\mu)$,
$\psisx(\mu;\gl)$ and $\psisxx(\mu;\gl)$ are continuous at $\mu=0$
with
$\psis(0)=\psisx(0;\gl)=\psisxx(0;\gl)=0$.
On the other hand, if $0\notin\cS$, then $\phis(0)=0$ and
$\psisx(\mu;\gl)\to-\infty$ while a simple calculation yields
$\psisxx(\mu;\gl)\to \frac 12 s\bigpar{\log(\gl s)-1}-\log(s!)$
as $\mu\to0$, where $s = \min\cS$.

\begin{theorem}\label{Tmax}
The following hold for every fixed $\gl>0$ and $j=1$ or $2$, except
for $j=2$ in the trivial case $|\cS|=1$.
  \begin{thmenumerate}
\item
$\egl$ is the set of stationary points of $\psisxj$, possibly
with $0$ added:
\begin{equation*}
  \egl\cap(0,\infty)
=\left\{\mu:\dddmu\psisxj(\mu;\gl)=0\right\}.
\end{equation*}
\item
$\eglo$ is the set of global maximum points of $\psisxj$:
\begin{equation*}
  \eglo
=\left\{\mu:\psisxj(\mu;\gl)=\max_{\mui\ge0}\psisxj(\mui;\gl)\right\}
.
\end{equation*}
  \end{thmenumerate}
\end{theorem}

\begin{proof}
\pfitem{i}
  Differentiation yields
  \begin{equation}\label{ps1}
\dddmu\psisx(\mu;\gl)
=
\frac{\phis'(\mu)}{\phis(\mu)}-\frac{\mu}{\gl}
  \end{equation}
and, after some simplifications,
  \begin{equation}\label{ps2}
\dddmu\psisxx(\mu;\gl)
=
\ddmu\Bigpar{\frac{\mu\phis'(\mu)}{2\phis(\mu)}}
\log\Bigpar{\frac{\gl\phis'(\mu)}{\mu\phis(\mu)}}.
  \end{equation}
By the \CSineq, provided $\mu>0$ and $|\cS|\ge2$,
\begin{multline}\label{ps2x}
\mu\ddmu{\biggpar{\frac{\mu\phis'(\mu)}{\phis(\mu)}}}
=
\mu\ddmu\biggpar{\frac{\sumjs j \mu^j/j!}{\sumjs \mu^j/j!}}  	
\\
=
\frac{\bigpar{\sumjs j^2 \mu^j/j!}\bigpar{\sumjs \mu^j/j!} 
  -\bigpar{\sumjs j \mu^j/j!}^2}
{\bigpar{\sumjs \mu^j/j!}^2}
>0.
\end{multline}
(See \refT{Tmonmu} below for a more general result.)
By \eqref{ps1}--\eqref{ps2x}, for $\mu>0$,
$\psisxj'(\mu;\gl)=0$ if
and only if 
$\phis'(\mu)/\phis(\mu)=\mu/\gl$, \ie, the characteristic equation
\eqref{chareq} holds.

\pfitem{ii}
By \eqref{ps1}--\eqref{ps2x} and \refL{Lmu}, $\psisxj$
is decreasing for large $\mu$. Furthermore, by the remarks prior to
the theorem, 
$\psisxj$ is either continuous at 0 or tends to $-\infty$ there.
This implies that $\psisxj$ has a finite maximum, attained at one or
several points in $[0,\infty)$. It remains to show that the maximum
points belong to $\egl$; it then follows that $\eglo$ equals the set
of maximum points.

If $\mu>0$ is a maximum point of $\psisxj$, then $\psisxj'(\mu)=0$ and
$\mu\in\egl$ by (i).

If $0$ is a maximum point and $0\in\cS$, then $0\in\egl$ by
definition.
Finally, if $0\notin\cS$, and $s\=\min\cS>0$,
then $\phis'(\mu)/\phis(\mu)\sim s/\mu\to\infty$ as $\mu\to0$, and it
follows from
\eqref{ps1}--\eqref{ps2x} that $\psisxj'>0$ for small $\mu$; hence
$0$ is not a maximum point in this case. (In fact, for $j=1$,
$\psisx(0)=-\infty$ when $0\notin\cS$, as remarked above.)
\end{proof}

We define $\muxmin(\gl):=\min\eglo$ and $\muxmax(\gl):=\max\eglo$; thus 
$|\eglo|=1$ 
(and \refT{T1} applies) if and only if $\muxmin=\muxmax$,
and in that case $\mux=\muxmin=\muxmax$.
We have defined $\mux$ only when $|\eglo|=1$; for convenience we
extend the definition to all $\gl>0$
by letting $\mux(\gl)$ by any element of $\eglo$
(for example $\muxmin(\gl)$ or $\muxmax(\gl)$).

\begin{corollary}\label{Cmax}
For every $\gl>0$ and $j=1,2$,
\begin{equation*}
\psis(\mux(\gl))
= \psisxj(\mux(\gl);\gl)
= \max_{\mu\ge0}  \psisxj(\mu;\gl).  
\end{equation*}
\end{corollary}

\begin{theorem}
  \label{TGA}
If $0<\gl_1<\gl_2$, then $\mux(\gl_1)\le\mux(\gl_2)$, with equality
only if $\mux(\gl_1)=\mux(\gl_2)=0$.
\end{theorem}
\begin{proof}
If $\mu<\muxmax(\gl_1)\in\eglox{\gl_1}$, then
$\psisx(\mu;\gl_1)\le\psisx(\muxmax(\gl_1);\gl_1)$ by \refT{Tmax}(ii),
and thus
\begin{equation*}
  \begin{split}
\psisx(\mu;\gl_2)
&=
\psisx(\mu;\gl_1)+\frac{\mu^2}{2}\Bigpar{\frac1{\gl_1}-\frac1{\gl_2}}
\\
&<\psisx(\muxmax(\gl_1);\gl_1)
+\frac{\muxmax(\gl_1)^2}{2}\Bigpar{\frac1{\gl_1}-\frac1{\gl_2}}
=
\psisx(\muxmax(\gl_1);\gl_2),
  \end{split}
\end{equation*}
so $\mu$ is not a global maximum point of $\psisx(\mu;\gl_2)$ and, by
\refT{Tmax}(ii) again, $\mu\notin \eglox{\gl_2}$.
Hence, $\mux(\gl_2)\ge\muxmin(\gl_2)\ge\muxmax(\gl_1)\ge\mux(\gl_1)$.
Equality is possible only if $\mu\=\mux(\gl_1)=\mux(\gl_2)\in
\eglx{\gl_1}\cap\eglx{\gl_2}$, and then the characteristic equation
\eqref{chareq} is satisfied with $\mu$ and both $\gl_1$ and $\gl_2$;
hence $\mu^2/\gl_1=\mu^2/\gl_2$ and $\mu=0$.
\end{proof}

\begin{theorem}\label{Tmono}
  \begin{thmenumerate}
\item
  For every $\gl>0$,
$\mux(\gl')\upto\muxmin(\gl)$ as $\gl'\upto\gl$
and
$\mux(\gl')\downto\muxmax(\gl)$ as $\gl'\downto\gl$.
\item
$\mux(\gl)\to0$ as $\gl\to0$.
\item
$\mux(\gl)\to\infty$ as $\gl\to\infty$.
  \end{thmenumerate}
\end{theorem}
\begin{proof}
\pfitem{i}
Let $\mu_0\=\lim_{\gl'\upto\gl}\mux(\gl')$; the limit exist by the
monotonicity in \refT{TGA}.
For any fixed $\mu$, $\psisx(\mu;\gl')\le\psisx(\mux(\gl');\gl')$ 
by \refT{Tmax}
and
it follows by continuity that
$\psisx(\mu;\gl)\le\psisx(\mu_0;\gl)$. Hence, by \refT{Tmax} again,
$\mu_0\in\eglo$, and \refT{TGA} implies that $\mu_0=\muxmin(\gl)$.

The second statement is proved similarly.

\pfitem{ii} This is similar.
Assume $\mu_0\=\lim_{\gl\to0}\mux(\gl)>0$, and let
$\mua\=\mu_0/2$. Then $\psisx(\mua;\gl)\le\psisx(\mux(\gl);\gl)$ for all $\gl$ 
by \refT{Tmax}
which contradicts the fact that,
by \eqref{psixdef},
\begin{align*}
\psisx(\mua;\gl)-\psisx(\mux(\gl);\gl)
&=\frac{\mux(\gl)^2-\mua^2}{2\gl}+\O(1)\\
&\ge\frac{\mu_0^2-\mua^2}{2\gl}+\O(1)
\to\infty
\end{align*}
as $\gl\to0$.

\pfitem{iii}
Assume $\mu_0\=\lim_{\gl\to\infty}\mux(\gl)<\infty$, and let
$\mua\=\mu_0+1$. Then $\psisx(\mua;\gl)\le\psisx(\mux(\gl);\gl)$ for all $\gl$ 
by \refT{Tmax}
and
it follows from \eqref{psixdef} by continuity, letting
$\gl\to\infty$, that
$\log\phis(\mua)\le\log\phis(\mu_0)$, a contradiction since $\phis$ is
strictly increasing.
\end{proof}

\begin{remark}
For large $\gl$, we have the estimates
$\cc \gl\qq\le\mux(\gl)\le\CC\gl$,
where the lower bound follows from \eqref{chareq}
 and the upper from
\eqref{hbnd}.
Examples \refand{EN}{E01} show that both these orders of growth can be
attained. 
\end{remark}

We see from \refT{Tmono} that $|\eglo|>1$ exactly when
$\mux(\gl)$ is discontinuous, and that all discontinuities are jump
discontinuities: $\mux$ jumps form $\muxmin(\gl)$ to $\muxmax(\gl)$. 
In accordance with \refR{Rphase}, we interpret
these discontinuities as \emph{phase transitions} of $\gnlns$. 
More generally, we say that we have a phase transition at each $\gl$
where $\mux$ is not analytic. (See, further, \refT{Tphase}.)
We show that there is only a countable number of phase transitions
with jump discontinuities, and we note from
\refE{Einfty} that the number may be infinite.

\begin{theorem}\label{Tcountable}
The set $\Lex\=\set{\gl>0:|\eglo|>1}
=\set{\gl:\muxmin(\gl)<\muxmax(\gl)}$ 
of $\gl$ such that
  \refT{T1} does not apply is at most countable.
\end{theorem}
\begin{proof}
  By \refT{TGA}, the open intervals
  $\bigpar{\muxmin(\gl),\muxmax(\gl})$, $\gl>0$, are disjoint, and thus
  at most a countable number of them are non-empty.
\end{proof}

It follows from \refT{TGA} and its corollaries that $\mux$ is
the \emph{inverse function} of a continuous non-decreasing function
with graph $\set{(\mu,\gl):\mu\in[\muxmin(\gl),\muxmax(\gl)]}$;
the exceptional set $\Lex$ consists of the values taken by this
function on the intervals where it is constant.

For $\mu>0$ 
we can 
rewrite the characteristic equation \eqref{chareq}
as
\begin{equation}\label{glx}
  \gl=\glx(\mu)\=\frac{\mu\phis(\mu)}{\phis'(\mu)}.
\end{equation}
Thus, $\egl=\set{\mu>0:\glx(\mu)=\gl}$ or
$\set{\mu>0:\glx(\mu)=\gl}\cup\set0$.  
Note that our assumption $\cS\neq\set0$ implies that $\phis'(\mu)>0$
for all $\mu>0$, so $\glx$ is well-defined.

\begin{lemma}\label{Lglx}
The function $\glx$ is positive and analytic on $(0,\infty)$, 
with $\lim_{\mu\to\infty}\glx(\mu)=\infty$ and
\begin{align*}
  \lim_{\mu\to0}\glx(\mu)=
  \begin{cases}
0,& 0\notin\cS,\\
0,& 0\in\cS,\,1\in\cS,\\
1,& 0\in\cS,\,1\notin\cS,\,2\in\cS,\\
\infty,& 0\in\cS,\,1\notin\cS,\,2\notin\cS.
  \end{cases}
\end{align*}
\end{lemma}
\begin{proof}
That $\glx$ is analytic and positive is evident.
By \refL{Lmu}, $\glx(\mu)\ge\cc\mu$ for large $\mu$, and
the behaviour as $\mu\to0$ follows by looking at the first non-zero
terms in the Taylor expansions of $\phis$ and $\phis'$.
\end{proof}

\begin{lemma}
  \label{Lpsi}
$\psis'(\mu)$ and $\glx'(\mu)$ have the same sign for every
  $\mu>0$. Hence $\psis$ and $\glx$ have the same stationary points
  and are increasing or decreasing on the same intervals.
\end{lemma}

\begin{proof}
  By \eqref{psidef} and \eqref{psixdef},
$\psis(\mu)=\psisx(\mu;\glx(\mu))$. Differentiating and using
  \eqref{ps1} we obtain, for all $\mu>0$,
  \begin{equation*}
\psis'(\mu)
=
\ddpmu\psisx(\mu,\glx(\mu))
+
\ddpgl\psisx(\mu,\glx(\mu))\glx'(\mu)
=0+\frac{\mu^2}{2\glx(\mu)^2}\glx'(\mu)	.
\qedhere
  \end{equation*}
\end{proof}

\begin{theorem}
  \label{Tlex}
\begin{thmenumerate}
  \item
If $\glx$ is decreasing on an interval $(\mu_1,\mu_2)$ with
$0\le\mu_1<\mu_2$, then there exists $\gl\in\Lex$ with
$\muxmin(\gl)\le\mu_1<\mu_2\le\muxmax(\gl)$.
\item
$\Lex=\emptyset$ if and only if $\glx$ is increasing on $(0,\infty)$.
In this case, for every $\gl>0$ either $\egl=\set{\mux(\gl)}$ 
with $\mux(\gl)\ge0$
or
$\egl=\set{0,\mux(\gl)}$
with $\mux(\gl)>0$.
\end{thmenumerate}
\end{theorem}

\begin{proof}
  \pfitem{i}
Suppose that $\muxmin(\gl)\in(\mu_1,\mu_2)$ for some $\gl>0$. Taking a
sequence $\gl_n\upto\gl$, we have $\mux(\gl_n)\upto\muxmin(\gl)$ by
\refT{Tmono}, so, for large $n$, $\mu_1<\mux(\gl_n)<\muxmin(\gl)<\mu_2$
and hence $\gl_n=\glx(\mux(\gl_n))>\glx(\muxmin(\gl))=\gl$, a
contradiction.
Hence, $\muxmin(\gl)\notin(\mu_1,\mu_2)$ for all $\gl>0$. 

Let $\glo\=\sup\set{\gl:\muxmin(\gl)\le\mu_1}$.
Thus, $\muxmin(\gl)\ge\mu_2$ for $\gl>\glo$. By
\refT{Tmono}(ii)(iii), $0<\glo<\infty$, and by \refT{Tmono}(i)
$\muxmin(\glo)\le\mu_1$ and $\muxmax(\glo)\ge\mu_2$.
In particular, $\muxmin(\glo)<\muxmax(\glo)$ so $\glo\in\Lex$.

\pfitem{ii}
If $\glx$ is not increasing, then $\glx'(\mu)<0$ for some $\mu>0$ and
(i) applies to some interval $(\mu-\eps,\mu+\eps)$ and shows that
$\Lex\neq\emptyset$. 

If $\glx$ is increasing, then $\psis$ is (strictly) increasing on
$(0,\infty)$ by \refL{Lpsi}; if $0\in\cS$, then $\psis$ is continuous
at 0 and thus increasing on $[0,\infty)$ also.
Consequently, $\egl$ contains a unique $\mux=\max\egl$ that maximizes
$\psis$.
Further, when $\glx$ is increasing, there is at most one positive
solution to $\gl=\glx(\mu)$, and thus to \eqref{chareq}, and the
result on $\egl$ follows.
\end{proof}

We next study whether $\mux=0$ is possible. Note that this is a
rather degenerate case, when \refT{T1} shows that $\gnlns$ is very
sparse with $\o_p(n)$ edges and $n_0/n\pto1$, which is
to say that $n(1-\o_p(1))$
vertices are isolated.

\begin{theorem}\label{Tmux0}
  \begin{thmenumerate}
\item
If\/ $0\notin\cS$, then $\mux>0$ for every $\gl>0$.	
\item
If\/ $0\in\cS$ and $1\in\cS$, then $\mux>0$ for every $\gl>0$.	
\item
If\/ $0\in\cS$ and $1\notin\cS$, 
then there exists $\gl_0>0$ such that
$\mux=0$ for every $\gl<\gl_0$, but 
$\mux>0$ for every $\gl>\gl_0$.
  \end{thmenumerate}
\end{theorem}

\begin{proof}
\pfitem{i}
Trivial, since $0\notin\egl$ in this case.

\pfitem{ii}
In this case, $\phis(0)=\phis'(0)=1$ and thus \eqref{ps1} shows that
$$
\frac d{d\mu} \psisx(\mu;\gl)\to1>0
$$
as $\mu\to0$;
hence $\psis(\mu;\gl)$ is
increasing for small $\mu$ and 0 is not a maximum point. By \refT{Tmax}(ii),
$0\notin\eglo$.

\pfitem{iii}
By \refL{Lglx},
$\glx(\mu)$ tends to $\infty$ as
$\mu\to\infty$, and to either $1$ or
$\infty$ as $\mu\to0$. Hence $\gl_1\=\inf_{\mu>0}\glx(\mu)>0$.
If $\gl<\gl_1$, there is thus no positive solution to \eqref{chareq},
so $\egl=\set0$ and $\mux=0$. The existence of $\gl_0\ge\gl_1$ as
asserted now follows from \refT{TGA} and \refT{Tmono}.
\end{proof}

In case (iii), $\muxmin(\gl_0)=0$ by \refT{Tmono};
it is possible both that $\eglox{\gl_0}=\set0$ so that
$\muxmax(\gl_0)=0$, 
and that $|\eglox{\gl_0}|>1$ so that $\glo\in\Lex$ and
$\muxmax(\gl_0)>0$. 
We can classify these subcases too.

\begin{theorem}
  \label{Tmux00}
Suppose that\/ $0\in\cS$ and $1\notin\cS$.
  \begin{thmenumerate}
\item
If\/ $2\in\cS$ and $3\notin\cS$
($\cS=\set{0,2,s,\dots}$ with $s\ge4$, or \set{0,2}), 
then 
$\mux(\gl)=0$ for $\gl\le\glo=1$ and $\mux(\gl)>0$ for $\gl>1$, with
$\muxmax(1)=0$ and thus $1\notin\Lex$ and $\mux(\gl)\downto0$ as
$\gl\downto1$. 
\item
If\/ $2\in\cS$ and $3\in\cS$, 
or if $2\notin\cS$
($\cS=\set{0,2,3,\dots}$ 
or \set{0,s,\dots} with $s\ge3$), 
then there exists $\glo>0$ such that
$\mux(\gl)=0$ for $\gl<\glo$ and $\mux(\gl)>0$ for $\gl>\glo$, with
$\muxmax(\glo)>0=\muxmin(\glo)$ and thus $\glo\in\Lex$ and 
$\lim_{\gl\downto\glo}\mux(\gl)>0$.
  \end{thmenumerate}
\end{theorem}

\begin{proof}
  \pfitem{ii}
If $2,3\in\cS$, then the Taylor series $\phis(\mu)=1+\frac12\mu^2+\dots$
and $\phis'(\mu)=\mu+\frac12\mu^2+\dots$ yield
$\glx(\mu)=1-\frac12\mu+\dots$ for small $\mu$. 
If $2\notin\cS$, then $1/\glx(\mu)\to0$ as $\mu\to0$ by \refL{Lglx}.
In both cases, $1/\glx$ is analytic in a neighbourhood of 0 and
increases on an interval $(0,\mu_0)$, so $\glx$ decreases there and
the result follows by
Theorems \ref{Tlex}(i) and \ref{Tmux0}(iii).
\pfitem{i}
Taylor expansions as in the proof of (ii) show that 
$\glx(\mu)=1+\frac13\mu^2+\dots$ (when $4\in\cS$) or $\glx(\mu)=1+\frac12\mu^2+\dots$
(when $4\notin\cS$) so $\glx$ increases for small $\mu$, say in an interval
$(0,\mu_0)$. By \refL{Lpsi}, $\psis$ increases in $(0,\mu_0)$,
and since $\psis$ is continuous at 0 we have $\psis(\mu)>\psis(0)$ for
$0<\mu<\mu_0$. 

However, we also have to consider larger $\mu$, and we use \refL{L024}
below which implies that $\gl_1\=\inf_{\mu\ge\mu_0}\glx(\mu)>1$.
It follows that if $\gl\le1$, then $\egl=\set0$, 
and in particular $\muxmax(1)=0$. Similarly,
if $1<\gl<\gl_1$,
then $\egl=\set{0,\mu}$ for the unique $\mu\in(0,\mu_0)$ with
$\glx(\mu)=\gl$; in this case, 
$\psis(\mu)>\psis(0)$ and we have $\mux=\mu>0$. Finally,
by \refT{Tmono}, 
$\mux(\gl)\downto\muxmax(1)=0$ as $\gl\downto1$.
\end{proof}

\begin{lemma}
  \label{L024}
Under the assumptions $0,2\in\cS$ and $1,3\notin\cS$ of
\refT{Tmux00}(i),
$\glx(\mu)>1$ for every $\mu>0$.
\end{lemma}

\begin{proof}
By \eqref{glx}, the claimed inequality is equivalent to
$\phis'(\mu)<\mu\phis(\mu)$, where 
$\phis'(\mu)=\sum_{k\in\cS} \mu^{k-1}/(k-1)!$. 
First, we use the trivial estimates
\begin{equation*}
  \phis'(\mu)\le\phi_{\bbZo\setminus\set{1,3}}'(\mu)
=e^\mu-1-\tfrac12\mu^2
\end{equation*}
and
\begin{equation*}
  \phis(\mu)\ge 1+\tfrac12\mu^2.
\end{equation*}
We may verify numerically (by \maple, or otherwise) that 
$e^\mu-1-\frac12\mu^2<\mu(1+\frac12\mu^2)$ for $0<\mu<3.38$, and thus the claim
holds in this range.

For larger $\mu$, we split the sum for $\phis'(\mu)$ into two
parts. With $K\=\floor{\mu^2}$, we have that
\begin{equation}\label{l024a}
  \begin{split}
  \sum_{k\in\cS,\;k\le K}&\frac{\mu^{k-1}}{(k-1)!}
\le \mu+\sum_{k\in\cS,\;4\le k\le K} \frac k{\mu}\cdot \frac{\mu^{k}}{k!}
\le \mu+\frac{K}{\mu}\sum_{k\in\cS,\; k\ge 4}  \frac{\mu^{k}}{k!}
\\&
=\mu+\frac{K}{\mu}\bigpar{\phis(\mu)-1-\tfrac12\mu^2}
\\&
\le\mu\phis(\mu)-\tfrac12\mu^3.	
  \end{split}
\end{equation}
For $k>K$ we use the Chernoff bound for the Poisson distribution,
see \eg{} \cite[Theorem 2.1 and Remark 2.6]{JLR}:
\begin{equation}\label{l024b}
  \begin{split}
\sum_{k\in\cS,\;k\ge K+1}\frac{\mu^{k-1}}{(k-1)!}
&\le
\sum_{k=K}^\infty\frac{\mu^{k}}{k!}	
=e^\mu\P\bigpar{\Po(\mu)\ge K}
\\&
\le
\exp\bigpar{\mu-K\log(K/\mu)+K-\mu}
=\parfrac{e\mu}{K}^K.
  \end{split}
\end{equation}
For $\mu\ge3$, $K/\mu=\floor{\mu^2}/\mu>9/\sqrt{10}>e$, so the sum in
\eqref{l024b} is less than 1. We combine this with \eqref{l024a} to obtain
$\phis'(\mu)<\mu\phis(\mu)$ for $\mu \ge 3$.
\end{proof}

\begin{theorem}
  \label{Tphase}
The set \set{\gl:\mux \text{ is not analytic at }\gl} of phase
transitions is at most countable.
Each phase transition is of one of the following types.
\begin{romenumerate}
  \item
A jump discontinuity: $\gl\in\Lex$ and $\muxmin(\gl)<\muxmax(\gl)$.
\item
A continuous phase transition with $\muxmin(\gl)=\muxmax(\gl)=0$ but
$\mux(\gl')>0$ for $\gl'>\gl$. This can happen only at $\gl=1$, where it
happens if and only if $0,2\in\cS$ but $1,3\notin\cS$.
\item
A continuous phase transition with $\mux>0$; in this case,
$\gl=\glx(\mu)$ for some $\mu>0$ with $\glx'(\mu)=0$ but $\glx$
increasing in a neighbourhood of $\mu$; thus $\mu$ is  an
inflection point of $\glx$.
\end{romenumerate}
\end{theorem}

Examples of type (i) and (ii) are given in \refS{Sex}.
We do not know whether (iii) actually occurs.

\begin{proof}
  Since $\glx(\mux(\gl))=\gl$ when $\mux(\gl)>0$ and $\glx$ is
  analytic, the implicit 
  function theorem shows  that $\mux$ is analytic at every point where
  it is continuous and positive and $\glx'(\mux)\neq0$.
Hence we have only the three given possibilities;
in  (iii), $\glx$ has to be increasing in a neighbourhood of $\mu$
  since otherwise we would have a jump discontinuity by
  \refT{Tlex}(i).

The further characterization in (ii) follows by Theorems
\refand{Tmux0}{Tmux00}.

The number of jump discontinuities is countable by \refT{Tcountable},
and so is the number of inflection points of $\glx$, while there is at
most one phase transition of type (ii).
\end{proof}

In Case (ii), by the proof of \refT{Tmux00},
$\glx(\mu)=1+c\mu^2+\o(\mu^2)$
as $\mu\to0$, so
$\mux(\gl)\sim c'\sqrt{\gl-1}$ as $\gl\downto1$ and we have a 
square-root type singularity. In Case (iii), provided it happens at all, if
the inflection point is $\muo$, then 
$\glx(\mu)-\glx(\muo)\sim c(\mu-\mu_0)^m$ as 
$\mu\to\muo$ for some odd $m\ge3$ (presumably
$m=3$), and thus $\mux(\gl)-\muo\sim c'(\gl-\glo)^{1/m}$ as
$\gl\to\glo=\glx(\muo)$. 

Theorems \refand{Tlex}{Tphase} show that phase transitions, except the
possible one of type (ii), occur when $\glx$ ceases to be increasing
at some points, or at least almost ceases to be, in the form of an
inflection point.
(Recall that $\glx\to\infty$, so it increases in the long run.)
We have no criterion for when this happens, but it seems likely
that it occurs whenever there are large gaps in $\cS$.

\begin{problems}
Several open problems remain. For example:
  \begin{romenumerate}
\item
Is there ever any phase transition of type (iii) in \refT{Tphase}
(with an inflection point of $\glx$)?
\item
Does the set of phase transitions lack accumulation points? In other
words, if there is an infinite number of phase transitions (as in
\refE{Einfty}), can we always order them in an increasing sequence
$\gl_1<\gl_2<\dots$ with $\gl_n\to\infty$?
\item
Is $|\eglo|$ always 1 or 2? In the latter case, is always $|\egl|=3$,
as in \refE{E03}, with one intermediate point that is a minimum rather
than a maximum of $\psisx$ and $\psisxx$?
 \end{romenumerate}
\end{problems}

\section{Monotonicity}

We begin with a general result that is a simple consequence of
standard results. Recall that if $X$ and $Y$ are two
random variables, we say that $X$ is \emph{stochastically smaller}
than $Y$, 
and write $X\lest Y$, if $\P(X> x)\le\P(Y>x)$ for every real $x$; it
is well-known that this is equivalent to the existence of a coupling
$(X',Y')$ of $(X,Y)$ with $X'\le  Y'$ a.s.
See, for example, \cite[Section IV.1]{Lindv}.

\begin{lemma}\label{Lmon}
Let $Y$ be a random variable on $\bbZo$ with a \pgf{} 
$\phi_Y(z)=\sum_{i\ge0} p_i z^i$
that is finite for all $z$, and let, for $\mu>0$, $Y_\mu$ have the
conjugate (or tilted) distribution $\P(Y_\mu=k)=p_k
\mu^k/\phi_Y(\mu)$.
\begin{romenumerate}
  \item
If $f:\bbZo\to\bbR$ is a non-decreasing function such that
$\E|f(Y_\mu)|<\infty$ for every $\mu>0$, then
$\ddmu\E f(Y_\mu)\ge0$, with strict inequality except in the trivial
case when $f$ is constant on \set{k:p_k>0}.
\item
If $\mu_1\le\mu_2$ then $Y_{\mu_1}\lest Y_{\mu_2}$.
\end{romenumerate}
\end{lemma}
\begin{proof}
\pfitem{i}
  \begin{equation*}
	\begin{split}
\mu\ddmu \E f(Y_\mu)	  
&=
\mu\ddmu \frac{\sum_kf(k)p_k\mu^k}{\sum_kp_k\mu^k}
\\&
=
\frac{\sum_kkf(k)p_k\mu^k}
{\sum_kp_k\mu^k}
-\frac{\bigpar{\sum_kf(k)p_k\mu^k}
\bigpar{\sum_kkp_k\mu^k}}
{\bigpar{\sum_kp_k\mu^k}^2}
\\&
=\E \bigpar{Y_\mu f(Y_\mu)}-\E (Y_\mu)\E \bigpar{f(Y_\mu)}
=\Cov\bigpar{f(Y_\mu),Y_\mu)}
\ge0,
	\end{split}
  \end{equation*}
since, as is well-known, 
the two non-decreasing functions $f(Y_\mu)$ and $Y_\mu$ 
of $Y_\mu$ are
positively correlated, for example by a calculation of 
$\E\big[(f(Y_\mu)-f(Y_\mu'))(Y_\mu-Y_\mu')\big]\ge0$ with $Y_\mu'$
an independent copy of $Y_\mu$.
The same proof yields strict inequality if $f(j)\neq f(k)$ for some
$j,k$ with $p_j,p_k>0$.
\pfitem{ii}
By (i), $\P(Y_\mu>x)$ is a non-decreasing function of $\mu$ for every
$x$.
\end{proof}

Let, for $\mu>0$, $\xmu\sim\Po(\mu)$ and $\xmus\sim\PoS(\mu)$.
Applying \refL{Lmon} to $X_{1,\cS}$, we obtain the following.
\begin{theorem}\label{Tmonmu}
\begin{thmenumerate}
  \item
If $f:\bbZo\to\bbR$ is a non-decreasing function such that
$\E|f(X_\mu)|<\infty$ for every $\mu>0$, then
\begin{equation*}
\ddmu\E f(\xmus)
=
\ddmu\E \bigpar{f(\xmu)\mid\xmu\in\cS}
\ge0,  
\end{equation*}
with strict inequality except in the trivial
case when $f$ is constant on $\cS$.
\item
If $\mu_1\le\mu_2$ then $X_{\mu_1,\cS}\lest X_{\mu_2,\cS}$.
\end{thmenumerate}
\end{theorem}
This shows, in conjunction with \refT{T1}, that the asymptotic degree
distribution of $\gnlns$ is stochastically increasing in $\mux(\gl)$
and thus, by \refT{TGA}, in $\gl$. In particular, the asymptotic edge
density, which is given by $\E X_{\mux(\gl),\cS}=\nu(\mux(\gl))$, 
is an increasing function of $\gl$
(except that it is constant in the trivial case $|\cS|=1$).
This holds for finite $n$ too.

\begin{theorem}
  \label{Tmone}
If\/ $0<p_1\le p_2< 1$, then $\edg(\gnxs{p_1})\lest\edg(\gnxs{p_2})$.
\end{theorem}
\begin{proof}
This is another application of \refL{Lmon}, since it follows from
  \eqref{gnps} that $\edg(\gnps)$ has the conjugate distribution
  $Y_{p/(1-p)}$ with $Y=\edg(\gnxs{\frac12})$.
\end{proof}

Unfortunately, if we consider the entire random graph $\gnps$ 
(and not just the number of its edges), it is
in general \emph{not} stochastically increasing in $p$.

\begin{example}
  Let $n=4$ and let $\cS=\set{0,2}$ (or the set of all even numbers).
The \sgraph{s} are, ignoring the labelling: 
(i) $E_4$, the empty graph with no edges, (ii)
$C_3+E_1$, a 3-cycle plus an isolated vertex, (iii) $C_4$, a 4-cycle.
We have $\P(\gnps=E_4)\to1$ as $p\to0$ and $\P(\gnps=C_4)\to1$ as $p\to1$.
Hence, if $f(G)$ is the number of 3-cycles in $G$, then
$\E f(\gnps)=\P(\gnps=C_3+E_1)$ tends to 0 both as $p\to0$ and $p\to1$, so
this expectation is not monotone in $p$.
\end{example}

\begin{problem}
  Is the random multigraph $\gnmuxs$ defined in \refS{Smulti}
  stochastically increasing in $\nu$?
(Its number of edges is, by the same argument as for \refT{Tmone}.)
\end{problem}

For the existence of a giant component, we note that the crucial
quantity $Q(\mu)$ in \eqref{Qdef} is \emph{not} always monotone in
$\mu$, not even in the classical case $\cS=\bbZo$ (when
$Q(\mu)=\mu^2-\mu$).
Nevertheless, the condition $Q(\mu)>0$ is monotone.

\begin{theorem}
  \label{Tgiantmono}
If $\mu_1\le\mu_2$ and $Q(\mu_1)>0$, then $Q(\mu_2)>0$.

Moreover, 
assuming $\cS\not\subseteq\set{0,2}$,
if
$\gl_1\le\gl_2$ and thus $\mux(\gl_1)\le\mux(\gl_2)$,
then $\xix(\mux(\gl_1))\ge\xix(\mux(\gl_2))$
and $\gammax(\mux(\gl_1))\le\gammax(\mux(\gl_2))$.
Hence, if $\gnxs{\gl_1/n}$  has a giant component, then so has
$\gnxs{\gl_2/n}$ for all $\gl_2\ge\gl_1$, and it is (asymptotically)
at least as large.
\end{theorem}

\begin{proof}
  The condition $Q(\mu)>0$ is equivalent to
  $\mu\phis''(\mu)/\phis'(\mu)>1$.
Since $\phis'(\mu)=\phi_\SI(\mu)$,
\begin{equation}\label{s-1}
  \frac{\mu\phis''(\mu)}{\phis'(\mu)}
=
  \frac{\mu\phi_\SI'(\mu)}{\phi_\SI(\mu)}
=
\E X_{\mu,\SI},
\end{equation}
which is non-decreasing by \refT{Tmonmu}(i).

The monotonicity of $\xix$ and $\gammax$ follows from the branching
process interpretations in \refR{RGW} together with the stochastic
monotonicity \refT{Tmonmu}(ii) (for both $\cS$ and $\SI$) and \refT{TGA}.
\end{proof}

\begin{remark}
  \label{RS-1}
\refT{Tgiant} and \eqref{s-1} yield the curious relation that,
provided $\cS\not\subseteq\set{0,2}$, $\gnls$ has a giant component
if and only if $\nu_{\SI}(\mux)\=\E X_{\mux,\SI}>1$, with $\mux=\mux(\gl)$
calculated for $\cS$. Cf.\ Remarks \ref{RGW} and \ref{RcoreGW}.
\end{remark}

It follows similarly from \refR{RcoreGW} that the existence and size
of a $k$-core, for any fixed $k\ge3$, is monotone in $\gl$.

\section{Examples}\label{Sex} 

\begin{example}\label{EN}
  $\cS=\bbZo$, the Erd\H os--R\'enyi random graph.
We have in this much studied case that $\phis(\mu) = e^\mu$.
The characteristic equation \eqref{chareq} becomes
$\mu = \mu^2/\lambda$, with solutions $\mu=0$, $\mu=\lambda$.
By \eqref{psidef}, $\psis(\mu) = \frac12\mu$,
so $\psis(0)<\psis(\lambda)$. 
Therefore, and in accordance with \refT{Tmux0}(ii),
$\mux=\lambda$, so that
the number $n_i$ of vertices of degree $i$ satisfies
$n_i/n \to \lambda^i e^{-\lambda}/i!$, $i\ge 0$. 
This is a simple instance of \refT{Tlex}(ii).
There is no phase transition of $\gnlns$.
We have $\psisx(\mu;\gl)=\mu-\mu^2/(2\gl)$ and
$\psisxx(\mu;\gl)=\frac12{\mu}\bigpar{\log(\gl/\mu)+1}$.
Details of
the application of \refT{Tgiant} to this well understood case
may be found in \cite{MR98}. Similarly, the application
of \refT{Tcore} is described in \cite{JL}.
\end{example}

\begin{example}\label{Es}
  $\cS=\set s$, where $s \ge 1$.
We have $\phis(\mu) = \mu^s/s!$ and the characteristic equation
\eqref{chareq} becomes $s=\mu^2/\lambda$, with
solution $\mux=\sqrt{s\lambda}$. 
However, in this case, the value
of $\mux$ is in fact immaterial, since $\PoS(\mu)$ is a point mass at
$s$ for every $\mu>0$. Moreover, the graph $\gnlns$ is a random
regular graph with all vertices of degree $s$.
It is immediate
that $Q(\mu)=s(s-2)$, and so there exists a giant component if
$s>2$, and not if $s=1$.
In fact, if $s=1$, the graph consists of isolated edges only,
while if $s\ge3$, it is well-known that the graph
 is
connected with probability tending to 1, see
\citet[Section VII.6]{bollobas}. In the remaining case $s=2$, 
the graph consists of cycles, of which the largest has a
length that divided by $n$ converges to some non-degenerate
distribution on $[0,1]$, see, \eg{}, \citet{ABT}; this is thus an
exceptional case where we do not have convergence in probability of
the proportion of vertices in the giant cluster,
as in \refT{Tgiant}.
\end{example}

\begin{example}\label{Eeven}
  $\cS=2\bbZo$, the even numbers. In this case,
  \begin{equation*}
	\phis(\mu)=\sumk \frac{\mu^{2k}}{(2k)!}=\cosh \mu.
  \end{equation*}
The characteristic equation \eqref{chareq} is
\begin{equation*}
  \frac{\mu\sinh\mu}{\cosh\mu}=\frac{\mu^2}{\gl},
\end{equation*}
so either $\mu=0$ or
\begin{equation}\label{d1}
\gl=\glx(\mu)=  \frac{\mu}{\tanh\mu}.
\end{equation}

Since $\glx(\mu)=\mu/\tanh\mu$ increases (strictly) from 1 to $\infty$ for
$\mu\in[0,\infty)$, it follows that: if $\gl\le1$, $\mu=0$ is the only
  solution, while if $\gl>1$, there is also a positive solution.
We have
\begin{equation*}
  \psis(\mu)=\log(\cosh\mu)-\thalf\mu\tanh\mu.
\end{equation*}
Therefore,
\begin{equation*}
  \psis'(\mu)=\frac{\sinh(2\mu)-2\mu}{4\cosh^2\mu}>0
\end{equation*}
for $\mu>0$.
Hence, $\psis(\mu)>\psis(0)$ for $\mu>0$, whence $\mux$ is the
unique positive solution of \eqref{d1} when $\gl > 1$, \cf{}
\refL{Lpsi}, \refT{Tlex}(ii) and \refT{Tmux00}(i).

We thus have a continuous phase transition at $\gl=1$ with $\mux(1)=0$;
there is a unique $\mux$ 
(and thus \refT{T1} applies)
for every $\gl$, and $\mux(\gl)$ is a continuous function, but it is not
differentiable at $\gl=1$.
This is the only phase transition, and $\Lex=\emptyset$.

The asymptotic edge density (i.e., the number of edges per
vertex, see \eqref{t1bb})
is
\begin{equation}
  \label{evennu}
\nu(\mux)=\mux\tanh\mux.
\end{equation}

Since $1\notin\cS$, \refT{Tgiant} shows that there is a giant 
component as soon as $\mux>0$, \ie, if $\gl>1$. In fact,
it is easily seen that 
$$
Q(\mu)= \mu \tanh\mu \left(\frac\mu{\tanh\mu} -1\right).
$$

One may study a random even subgraph of a general graph $G$. It turns out
that the
random even subgraph with parameter $p\in[0,\frac12]$ is
related to the \rc\ model on $G$ with edge-parameter $2p$ and
cluster-weighting factor $q=2$. When $G$ is a planar graph, the random
even subgraph  
may be identified as the dual graph of the $+/-$ boundary of the
Ising model on the (Whitney) dual graph of $G$ with an appropriate
parameter-value. This relationship is especially fruitful when
$G$ is part of a planar lattice such as the square lattice $\bbZ^2$. 
See \cite{G-RC} for a general account of the
\rc\ model, and \cite{GJ} for its relationship with the random
even subgraph and the Ising model. 
\end{example}

\begin{example}\label{Eodd}
  $\cS=\set{1,3,5,\dots}$, the odd numbers. This time, $\phis(\mu) =\sinh \mu$,
and the characteristic equation is
$$
\frac{\mu\cosh\mu}{\sinh\mu} = \frac{\mu^2}{\gl}
$$
with $\glx(\mu)=\mu\tanh\mu$.
Since $0 \notin \cS$ and $\glx$ is increasing, the unique solution
$\mux$ is given as 
the unique positive solution of $\mu \tanh \mu = \gl$.
Cf.\
Theorems \ref{Tmux0}(i) and \ref{Tlex}(ii).
There is no phase transition.
This time,
$$
Q(\mu)= \frac\mu {\tanh\mu} (\mu \tanh\mu -1).
$$
Thus $Q(\mux)>0$ if and only if $\mux\tanh\mux>1$; since
$\mux\tanh\mux=\gl$, it follows that
there is a giant component for $\gl>1$, and not for $\gl\le1$.
(This also follows from \refR{RS-1} and \eqref{evennu}.)
In the critical case $\gl=1$ we have $\mux\tanh\mux=1$ and numerically 
$\mux\approx 1.19968$ 
and (asymptotic)
edge density 
$\nu(\mux)=\mux^2/\gl=\mux^2\approx 1.43923$.  
\end{example}

\begin{example}\label{E123xxx}
  $\cS=\set{1,2,3,\dots}=\bbZi$, graphs without isolated vertices.
We have that $\phis(\mu)=e^\mu-1$, and the characteristic equation is
$$
\frac{\mu e^\mu}{e^\mu -1} = \frac{\mu^2}{\gl}.
$$
Since $0\notin \cS$, we seek strictly positive solutions,
which is to say that $\gl=\glx(\mu)=\mu(1-e^{-\mu})$.
Since $\mu(1-e^{-\mu})$
is increasing on $(0,\infty)$, there is a unique such solution $\mux$
for every $\gl>0$. Cf.\ \refT{Tlex}(ii).

We have that
$$
Q(\mu) = \frac{\mu(\mu-1)e^\mu}{e^\mu-1}.
$$
Thus $Q(\mux)>0$ if and only
if $\mux>1$, which is to say that $\gl>\glx(1)=1-e^{-1}$. There is a giant
component when $\gl> 1-e^{-1}$, and not when $\gl\le 1-e^{-1}$.
In the critical case $\gl=1-e^{-1}$, $\mux=1$ and the critical 
(asymptotic) 
edge
density is 
$\nu(\mux)=\mux^2/\gl=e/(e-1)\approx1.58198$.  

\end{example}

\begin{example}\label{E01}
  $\cS=\set{0,1}$, matchings.
We have that $\phis(\mu) = 1+\mu$, and the characteristic equation
is
$$
\frac \mu{1+\mu} = \frac {\mu^2}{\gl}.
$$
Either $\mu=0$ or $\gl=\glx(\mu)=\mu(1+\mu)$, so the solutions
for given $\gl$ are $\mu=0$ and $\mu= -\frac12 + \sqrt{\gl+\frac14}$.
Since $\glx$ is increasing, \refT{Tlex}(ii) applies and shows that
$\mux=-\frac12 + \sqrt{\gl+\frac14}$ for all $\gl>0$.
This can also easily be verified directly, using
$$
\psis(\mu) =\log(1+\mu) -\frac{\mu}{2(1+\mu)}
$$
which yields $\psis'(\mu)>0$ for $\mu\ge 0$ (\cf{} \refL{Lpsi}).

By \refT{T1}, as $\ntoo$,
$$
\frac{n_0}n \pto 
\PoS(\mux)\set0 
= 
\frac1{\phis(\mux)}
=
\frac 1{1+\mux} 
= \frac1\gl\left\{\sqrt{\gl+\frac14}-\frac12\right\}.
$$

Obviously there is no giant component. Indeed,
$Q(\mu)=-\PoS(\mu)\set1<0$ for $\mu>0$.
\end{example}

\begin{example}\label{E02}
  $\cS=\set{0,2}$, isolated cycles.
We have that $\phis(\mu)= 1 + \frac12 \mu^2$, and the characteristic
equation is
$$
\frac{\mu^2}{1+\frac12\mu^2} = \frac{\mu^2}{\gl}.
$$
Therefore, either $\mu=0$ or 
$\gl=\glx(\mu)=1+\frac12\mu^2$, so that the solutions for a given
$\gl$ are
$\mu=0$ and, when $\gl>1$, $\mu=\sqrt{2(\gl-1)}$.
Again, $\glx$ is an increasing function, and so is
$$
\psis(\mu) = \log(1+\tfrac12\mu^2) - \frac{\frac12\mu^2}{1+\frac12\mu^2},
$$
by \refL{Lpsi} or direct calculations.
Thus, see \refT{Tlex}(ii),
$$
\mux = \begin{cases} 0 &\text{when } \gl \le 1,\\
\sqrt{2(\gl-1)} &\text{when } \gl >1.
\end{cases}
$$
We thus have a continuous phase transition at $\gl=1$, of the same
type as in \refE{Eeven}, see \refT{Tmux0}(iii) and \refT{Tmux00}(i).
There is no other phase transition.

It is easily seen that $Q(\mu)=0$ for all $\mu$, which may be interpreted
as saying that the random graph is, in a certain sense, critical
whenever $\gl>1$. If we remove the isolated vertices, and condition on
the number of remaining vertices, we obtain a
random regular graph with degree 2.
Hence, 
for $\mux>0$, \ie{}, for $\gl>1$, we see that
the largest component behaves as for $\cS=\set2$, see \refE{Es}, with
convergence of $\ver(\Gnlns)/n$ to a distribution but not to a constant.
\end{example}

\begin{example}\label{E03} 
  $\cS=\set{0,3}$.
This time, $\phis(\mu) = 1+\frac16 \mu^3$, and the characteristic equation is
$$
\frac{\frac12\mu^3}{1+\frac16\mu^3} = \frac{\mu^2}{\gl}.
$$
Either $\mu=0$, or
\begin{equation}\label{final}
\gl =
\glx(\mu)=
\frac{1+\frac16\mu^3}{\frac12\mu}.
\end{equation}
This is a convex function of $\mu$ with a minimum of $3^{2/3}$ at the 
point $\mu=3^{1/3}$. Hence, the
characteristic equation has no positive root when $\gl<3^{2/3}$, one
such root if 
$\gl=3^{2/3}$, and two such roots if $\gl>3^{2/3}$.
We have
$$
\psis(\mu) = \log(1+\tfrac16\mu^3) - \frac{\frac14\mu^3}{1+\frac16\mu^3}.
$$
Unlike the previous examples, $\psis$ is \emph{not} monotone, 
\cf{} \refL{Lpsi}.
In fact, 
$$
\psis(3^{1/3}) = \log(1+\tfrac12) -\tfrac12 <0=\psis(0),
$$
so the correct root is $\mux=0$ (rather than $3^{1/3}$)
when $\gl=3^{2/3}$. The function $\psi$ has a minimum at $\mu=3^{1/3}$;
$\psi$ decreases on $[0,3^{1/3}]$ and increases on $[3^{1/3},\infty)$.
There exists thus a unique $\muo>3^{1/3}$ such that
$\psis(\muo)=0$, and we set $\glo=2(1+\frac16\muo^3)/\muo$.
(Numerically, $\muo\approx2.03134$    
and $\glo=\glx(\muo)\approx 2.36002$.)  
We deduce that $\mux=0$ for $\gl<\glo$
while, for $\gl > \glo$, $\mux$ is the largest root of \eqref{final}.
For $\gl=\glo$, there are two roots $\mu$
of \eqref{final} with the same value of $\psis(\mu)$, 
so we have a jump phase transition 
and Theorem \ref{T1} does not apply; see Theorems \ref{Tmux0}(iii) and
\ref{Tmux00}(ii). 
There is no other phase transition, and $\Lex=\set{\glo}$.

Since $1\notin\cS$, by \refT{Tgiant}(ii),
there exists a giant component whenever $\gl>\glo$.
Indeed,
$$
Q(\mu) = \frac{\frac12\mu^2}{1+\frac16\mu^3}>0
$$
for every $\mu>0$.
\end{example}

\begin{example}
  \label{Einfty}
$\cS=\set{1,2,4,8,\dots}=\set{2^j:j\ge0}$.  
We claim that as $j\to\infty$, 
\begin{equation}\label{gw}
  \glx(2^jx)=2^{j}x^2\bigpar{1+\o(1)},
\end{equation}
for every $x\in(2/e,4/e)$. (In fact, this holds uniformly on every
closed subinterval of $(2/e,4/e)$.) It follows that if $a=4/e$ and
$\eps>0$ is small and fixed, then for large $j$, 
$\glx((a-\eps)2^j)\approx(a-\eps)^22^j$ and
$\glx((a+\eps)2^j)=\glx((a/2+\eps/2)2^{j+1})
\approx(a+\eps)^22^{j-1}$, so $\glx$ drops by a factor of about 2 in
the vicinity of $a2^j$. Consequently, for all large $j$,
there is an interval $I_j\subset\bigpar{(a-\eps)2^j,(a+\eps)2^j}$
where $\glx$ is decreasing, and thus by \refT{Tlex} there exists
$\gl\in\Lex$ such that $\muxmax(\gl)\ge\max I_j\ge 2^j$.
Hence the set $\set{\muxmax(\gl):\gl\in\Lex}$ is unbounded and thus
infinite, so $\Lex$ is infinite and there is an infinite number of
phase transitions.

To verify \eqref{gw} we show that if $\mu=2^jx\in(a2^{j-1},a2^j)$
then $\phis(\mu)=\sum_{k\in\cS}\mu^k/k!$  and 
$\phis'(\mu)=\sum_{k\in\cS}k\mu^{k}/(k-1)!$ 
are dominated by the terms with $k=2^j$:
\begin{align}\label{sjw}
  \phis(\mu)=\ioi\frac{\mu^{2^{j}}}{2^j!},
&&&
  \phis'(\mu)=\bigpar{1+\o(1)}\frac{\mu^{2^{j}-1}}{(2^j-1)!}
\end{align}
as $j\to\infty$;
this yields that $\phis'(\mu)\sim (2^j/\mu)\phis(\mu)$ and 
\begin{equation*}
  \glx(\mu)
= \frac{\mu}{\phis'(\mu)/\phis(\mu)}
\sim
\frac{\mu}{2^j/\mu}
=
\frac{\mu^2}{2^j}
=2^j x^2.
\end{equation*}
Finally, to show \eqref{sjw}, we observe by Stirling's
inequality that, as $k\to\infty$,
\begin{equation*}
  \frac{\mu^k}{k!}
=
\ioi(2\pi k)\qqw\parfrac{e\mu}{k}^k
\end{equation*}
and thus, with $k_i=2^i$,
\begin{align*}
\frac{\mu^{k_{i+1}}}{k_{i+1}!} \Bigm/
\frac{\mu^{k_{i}}}{k_{i}!} 
&=
\ioi 2\qqw\parfrac{e\mu}{2k_i}^{2k_i} \parfrac{e\mu}{k_i}^{-k_i}\\
&=
\bigpar{2\qqw+\o(1)} \parfrac{e\mu}{4k_i}^{k_i},
\end{align*}
which for $\mu=2^jx\in(a2^{j-1},a2^j)$
is exponentially small if $i\ge j$ and exponentially large if
$i<j$. The estimate \eqref{sjw} for $\phis(\mu)$ follows, and a
similar calculation with $k_i=2^i-1$ yields the result for $\phis'(\mu)$.
\end{example}

\section{Multigraphs}\label{Smulti}

As explained at the end of \refS{S1.5}, 
we shall count multigraphs with certain properties,
and shall later relate our conclusions to simple graphs.
Let $\gxn$ be the (infinite) set of all multigraphs on the vertex set
\set{1,2,\dots,n}, 
and let $\gxns$ be the subset of \smultigraph{s} on \set{1,2,\dots,n}
(we extend the definitions above to multigraphs in the obvious way,
noting that a loop counts two towards 
the degree of the vertex in question).

Let $\nu\ge 0$.
We define a random multigraph $\gnmux$ by 
taking $\Po(\nu)$ edges between each pair of vertices and $\Po(\frac12\nu)$
loops at each vertex, these random numbers being independent
of one another.
It is easily seen that this is equivalent to
assigning to each multigraph $G\in\gxn$ the probability
\begin{equation}\label{gnmux}
  \P(\gnmux=G)
=w(G)\nu^{e(G)} e^{- n^2\nu /2},
\end{equation}
where 
$$w(G)\=2^{-\ell}\prod_{j\ge2}j!^{-m_j},
$$
with $\ell$ the
number of loops of $G$, and $m_j$ the number of $j$-fold multiple
edges (including multiple loops). That is, $m_j=a_j+b_j$
where $a_j$ is the number of distinct pairs of vertices joined by
exactly $j$ parallel edges, and $b_j$ is the number of vertices having
exactly $j$ loops. See, \eg, \citet{JKLP}. 

Note that the total number of edges $\edg(\gnmux)$ is Poisson-distributed 
with parameter $\binom n2\nu+\frac12 n\nu=\frac12 n^2\nu$. 
We further define the random \smultigraph{} $\gnmuxs$ by conditioning
$\gnmux$ on being an \smultigraph.
Thus, for any multigraph $G\in\gxns$, by \eqref{gnmux},
\begin{equation}\label{b0}
  \P\bigpar{\gnmuxs=G}
= \frac 1 \zxnmus w(G)\nu^{e(G)},
\end{equation}
where
\begin{equation}\label{b1}
\zxnmus
\=
\sum_{G\in\gxns}w(G)\nu^{e(G)}
=e^{n^2\nu/2}\P(\gnmux\text{ is an \smultigraph}).
\end{equation}
We shall assume, of course, that $\gxns\neq\emptyset$. It is easy to
see that this holds for all $n$
if $\cS$ contains some even number,
but if all elements of $\cS$ are odd, then $n$ has to be even.
We tacitly assume this in the sequel.

If the multigraph $G\in\gxn$ is simple, \ie{} has no loops and no 
multiple edges, then $w(G)=1$ and \eqref{gnmux} yields
$\P(\gnmux=G)\asymp\nu^{\edg(G)}\asymp\P(\gnp=G)$
when $\nu=p/(1-p)$, \ie{}, $p=\nu/(1+\nu)$.
Hence, assuming this relation between $\nu$ and $p$, $\gnmux$
conditioned on being simple has the same distribution as $\gnp$.
Conditioning further on being \sgraph{s}, we obtain the following.
\begin{lemma}
  \label{Lsimple}
If\/ $\nu=p/(1-p)$, then 
\begin{equation*}
  \gnps
\eqd
\bigpar{\gnmuxs\mid\gnmuxs \text{\rm\ is simple}}.
\end{equation*}
\end{lemma}
We are interested in the case $np\to\gl<\infty$, and note that
$np\to\gl$ and $n\nu\to\gl$ are equivalent.

We shall also use 
the configuration model for random multigraphs with given vertex
degrees introduced
by \citet{BB80},
see  \citet[Section II.4]{bollobas}.
(See \citet{BenderC} and \citet{WormaldPhD,Wormald81} for related
arguments.)
To be precise, let us fix the vertex degrees to be some non-negative
integers $d_1,d_2,\dots,d_n$ (assuming tacitly that $\sum_id_i$ is
even); equivalently, we fix a degree sequence $\dd=\dn$.
We attach $d_i$ \emph{half-edges} (or \emph{stubs}) to vertex $i$.
The total number of half-edges is thus $2N\=\sumin
d_i$, and a \emph{configuration} is one of the
$(2N-1)!!=(2N)!/(2^NN!)$ partitions of the set of half-edges into $N$
pairs.
Each configuration defines a multigraph in $\gxn$ by combining each
pair of half-edges to an edge; this multigraph has vertex
degrees $d_1,\dots,d_n$ and $N=\frac12\sum_id_i$ edges.
By taking a uniformly random configuration we thus obtain a random
multigraph $\gnddx$ with the given degree sequence \dd.

It is easily seen that every multigraph $G\in\gxn$ with
the given vertex degrees $d_1,\dots,d_n$ arises from exactly
$w(G)\prodin d_i!$ configurations.
We obtain therefore 
that the contribution to $\zxnmus$ in \eqref{b1} from a set of
multigraphs with given vertex degrees $d_1,\dots,d_n\in\cS$ is given by 
summing $\nu^N/\prodin d_i!$ over all corresponding configurations.
In particular, since the number of configurations is $(2N-1)!!$, the
contribution to $\zxnmus$ from all multigraphs with vertex 
degrees $d_1,\dots,d_n\in\cS$ equals
\begin{equation}
  \label{a5}
\frac{(2N-1)!!\,\nu^{N}}{\prodin d_i!}.
\end{equation}
Moreover, for given $\dd=\dn$, the factor $\nu^N/\prodin d_i!$
is a constant, so by \eqref{gnmux}, the probability 
that $\gnmux$ belongs to any given set of
multigraphs with this degree sequence $\dd$ is proportional to the number of
corresponding configurations.
Consequently, if $d_i(G)$ denotes the degree of vertex $i$ in a
(multi)graph $G$, and $\dd(G)\=(d_i(G))_{i=1}^n$ for $G\in\gxn$, we
obtain the following well-known fact.
(This is another reason for the weights $w(G)$ in \eqref{gnmux}.)

\begin{lemma}
  \label{Lconf}
For any given degree sequence $\dd=\dn$ and any $\nu>0$, 
the random multigraph $\gnmux$ conditioned on having degree sequence
$\dd$ has the distribution given by the configuration model; in other
words,
\begin{equation*}
 \bigpar{\gnmux\mid\dd(\gnmux)=\dd}\eqd\gnddx. 
\end{equation*}
As a consequence, if every $d_i\in\cS$, the same holds for $\gnmuxs$.
\end{lemma}

We are interested in the case $\nu=\gl_n/n$, with $\gl_n\to\gl>0$.

\begin{theorem}
  \label{T1x}
The results of \refT{T1}(i)--(iii) 
hold with $\gnlns$ replaced by
$\gxnlns$.
Furthermore, 
\begin{equation}
  \label{zx}
\frac1n \log \P(\gxnln \text{\rm\ is an \sgraph})
\to \psi_\cS(\mux)-\tfrac12\gl
\end{equation}
and
$n^{-1} \log\zxnlns\to\psis(\mux)$. 
\end{theorem}

We will prove \refT{T1x} in the following section, and then obtain 
\refT{T1} as a consequence using \refL{Lsimple} and the following
technical result.  
\begin{lemma}
  \label{L1xy}
If\/ $\gl_n\to\gl>0$, then\/
$
\liminf_\ntoo
\P\bigpar{\gxnlns\text{\rm\ is simple}}>0.
$
\end{lemma}

\section{Proof of \refT{T1x}}\label{SpfT1x}

For notational convenience, 
we shall consider only the case $\gl_n=\gl$
for all $n$, while noting that our estimates may be extended to the
general case 
$\gl_n\to\gl$.
(The ``constants'' below generally depend on $\gl$, but they may be chosen
uniformly for $\gl$ lying in any compact subset of $(0,\infty)$.
Uniformity as $\gl\to0$ is less obvious, and perhaps not always true, 
but it is remarked a few times when it is important for later proofs.)

Let $\cNSn$ denote the set of all $\nn=(n_0,n_1,\dots)\in\cNS$ such that: 
$\sum_j n_j=n$ and $\sum_j jn_j$ is even. 
We write $\ann$ for the contribution to $\zxnls$ from all multigraphs
with $n_j$ 
vertices of degree $j$, which is to
say that
\begin{equation}\label{a8b}
  \zxnls=\sum_{\nn\in\cNSn}\ann
\end{equation}
and
\begin{equation}\label{a8c}
\P\bigpar{\nn(\gxnls)=\nn}=\frac{\ann}{\zxnls},
\qquad \nn\in\cNSn.
\end{equation}
By \eqref{a5} 
with $N=\frac12\sum_j jn_j$ and $\nu=\gl/n$,
\begin{equation}\label{a8}
  \ann
=
\frac{n!\,\bigpar{\sum_j jn_j-1}!!\,}{\prod_j n_j!\,\prod_j j!^{n_j}}
\parfrac{\gl}{n}^{N}.
\end{equation}
We note by Stirling's formula that
\begin{equation}\label{a5s}
(2N-1)!!
=\frac{(2N)!}{2^N N!} 
= \Bigpar{\frac{2N}{e}}^N\bigpar{\sqrt2+\O(N^{-1})},
\end{equation}
and it is easily verified that
\begin{equation}\label{a5b}
(2N-1)!!
\ge
 \Bigpar{\frac{2N}{e}}^N, \qquad N \ge 0.
\end{equation}

Let $\nnz=\nnz(n)$ be a mode of the random sequence $\nn(\gxnls)$, \ie{}, by
\eqref{a8c}, a sequence in $\cNSn$ that maximizes $\ann$. (In the case of
a tied maximum we make an arbitrary choice.) We write
$\Nz=\Nz(n)\=\frac12\sum_j j\nz_j$. 

We begin with a coarse but useful quantitative estimate,
obtained by considering only regular multigraphs.
Let $\ee_s\=(\gd_{is})_{i=0}^\infty$, where $\gd_{is}$ is the Kronecker delta.
In the following lemma we take an even number $s \in \cS$, if 
$\cS$ contains such a number. If not,
we pick an odd 
$s\in\cS$ and must then, as noted in the introduction,
restrict ourselves to even values of $n$.

\begin{lemma}\label{L1}
Let $s\in\cS$, and assume that $s$ is even if possible. 
Then 
\begin{equation}\label{a7-}
  \cc(s)^n\le z(n\ee_s)\le\annz\le\zxnls\le\CC^n.
\ccdef\ccli \CCdef\CCli
\end{equation}
As a consequence,
$$
\P(\gxnl\text{\rm\ is an \smultigraph})\ge \cc^n, \ccdef\ccavii
$$
and,
for any set $\cH$ of multigraphs,
\begin{equation}
  \label{a7}
 \P(\gxnls\in\cH)
\le \CC^n
 \P(\gxnl\in\cH).
\CCdef\CCavii
\end{equation}
\end{lemma}

\begin{proof}
By \eqref{a8} and
  \eqref{a5b}, 
  \begin{equation*}
z(n\ee_s)=
(s!)^{-n}(ns-1)!!\,\parfrac{\gl}n^{ns/2}
\ge	
(s!)^{-n}\Bigpar{\frac{ns}{e}\cdot\frac\gl{n}}^{ns/2},
  \end{equation*}
which yields the first inequality in \eqref{a7-} with 
$\ccli=(s!)^{-1}\xpar{\xfrac{s\gl}{e}}^{s/2}$.
The second and third inequalities are trivial, and the fourth follows
from \eqref{b1} (with $\nu=\gl/n$), 
which yields 
$\zxnls\le e^{n^2\nu/2}=e^{n\gl/2}$.

By 
\eqref{b1} and \eqref{a7-}, 
\begin{equation*}
 \P(\gxnl\text{ is an \smultigraph})
=e^{-\gl n/2}\zxnls
\ge \ccavii^n,
\end{equation*}
with $\ccavii=\ccli e^{-\gl/2}$. Consequently, \eqref{a7} follows with
$\CCavii=\ccavii\qw$ by the definition of conditional probabilities.
\end{proof}


\begin{lemma}\label{L2}
  There exists a constant $B=B(\cS,\gl)$ such that
  \begin{equation*}
\sum_{\nn\in\cNSn:\,N>Bn}	\ann < e^{-n}\annz \le e^{-n}\zxnls,
  \end{equation*}
where $N=\half\sum_i i n_i$.
Hence,
$\P\bigpar{\edg(\gxnls)>Bn}<e^{-n}$. 

More generally, 
for any $x\ge Bn$, 
$$\P\bigpar{\edg(\gxnls)>x}<e^{-x/B}.$$
Moreover, for any $\gl_0>0$, the constant $B$ can be chosen uniformly
for all $\gl\le\gl_0$. 
\end{lemma}

\begin{proof}
By \eqref{gnmux},
  \begin{equation}\label{l2a}
\sum_{\nn\in\cNSn:\,N>x} \ann 
\le
\sum_{\nn\in\cN^n:\,N>x} \ann 
= e^{n\gl/2}\P\bigpar{\edg(\gxnl)>x}.	
\end{equation}
Since the number of edges $\edg(\gxnl)\sim\Po(\frac12 \gl n)$, it follows by
standard Chernoff estimates for the Poisson distribution, see \eg{}
\cite[Corollary 2.4 and Remark 2.6]{JLR}, that if $B\ge4\gl$
and $x\ge Bn\ge4\gl n$, then 
\begin{equation}\label{l2b}
\P\bigpar{\edg(\gxnl)>x}	
=
\P\bigpar{\Po(\tfrac12\gl n)>x}	
< e^{-x}.
  \end{equation}
We choose
$$
B\ge\max \set{4\gl,\, \tfrac12\gl+1-\log\ccli(s)},
$$
and find by \eqref{l2a}, \eqref{l2b} and \eqref{a7-},
since $(B-1)n\le (B-1)x/B=x-x/B$,
  \begin{equation*}
\sum_{\nn\in\cNSn:\,N>x} \ann 
<
e^{n\gl/2-x-n\log \ccli(s)}\annz
\le
e^{(B-1)n-x}\annz
\le
e^{-x/B}\annz.
\end{equation*}
The results follow by this and \eqref{a8c}.
\end{proof}

Let $\nn\in\cNS$, and let $j$ and $k$ be two different indices in
$\cS$ such that $n_k\ge2$, and define $\nny\in\cNS$ by 
$\ny_j=n_j+2$,  $\ny_k=n_k-2$, and $\ny_i=n_i$ for $i\neq j,k$; 
in other words, we replace two
vertices of degree $k$ by vertices of degree $j$.
By \eqref{a8}, with $N=\frac12 \sum_i i n_i$ and $N'=\frac12\sum_i i\ny_i
=N+j-k$,
\begin{equation*}
  \begin{split}
\frac{\anny}{\ann}
&=
\frac{(2N'-1)!!}{(2N-1)!!}\cdot
\frac{n_j!\,n_k!}{n_j'!\,n_k'!}\cdot
\frac{k!^2}{j!^2}\cdot
\parfrac{\gl}{n}^{N'-N}
\\&
=
\frac{n_k(n_k-1)k!^2}{(n_j+1)(n_j+2)j!^2} (2N)^{j-k}
\parfrac{\gl}{n}^{j-k}\lrpar{1+\O\xpar{|j-k|^2/N}}.
  \end{split}
\end{equation*}
For $\nn=\nnz$ and any $j,k\in\cS$ with $\nz_k\ge2$, this quotient is
$\le 1$. Hence, for all $j,k\in\cS$ (also, trivially, if $\nz_k<2$ or $j=k$),
\begin{equation}
  \label{a9}
\nz_k(\nz_k-1)k!^2 \le (\nz_j+1)(\nz_j+2)j!^2\,
\parfrac{2\gl \Nz}{n}^{k-j}\lrpar{1+\O\xpar{|j-k|^2/\Nz}}.
\end{equation}
Furthermore, in the case $k>j$, we have $N'<N$ and
$$(2N-1)!!<(2N'-1)!!\,(2N)^{k-j},
$$
and we obtain in the same way
the sharper inequality
\begin{equation}
  \label{a9b}
\nz_k(\nz_k-1)k!^2 \le (\nz_j+1)(\nz_j+2)j!^2\,
\parfrac{2\gl \Nz}{n}^{k-j}.
\end{equation}

By \refL{L2},
\begin{equation}
  \label{a9c}
\Nz\le Bn.
\end{equation}
Now let \ntoo. 
Since $\Nz/n$ is bounded by \eqref{a9c}, each
subsequence has a subsequence such that $\Nz/n$ converges. Consider
such a subsequence, and assume that $2\Nz/n\to\nu\ge0$.
Furthermore, let $\pz_j\=\nz_j/n$. Then $\ppz\=(\pz_j)\ooo$ is a probability
distribution (the distribution of the degree of a random vertex in a
graph $G$ with $\nn(G)=\nnz$). Since the mean of this distribution is
$\sum_j j\pz_j=2\Nz/n$, which is bounded by \eqref{a9c} as \ntoo, this sequence
of distributions is tight, and by taking a further subsequence we may
assume that the distributions converge, \ie{} that $\nz_j/n\to \pq_j$ 
for some probability distribution $(\pq_j)\ooo$ and every $j\ge0$.
Clearly, this probability distribution is supported on $\cS$ in that
$\pq_j=0$ when $j\notin\cS$.

We treat the cases $\nu>0$ and $\nu=0$ separately. Assume first that $\nu>0$. 
Divide \eqref{a9} by $n^2$ and let \ntoo{} to find that
\begin{equation*}
\pq_k^2k!^2 \le \pq_j^2j!^2\,
\xpar{\gl \nu}^{k-j}, \qquad j,k\in\cS,
\end{equation*}
and thus
\begin{equation}
  \label{a11a}
\pq_k^2k!^2 \xpar{\gl \nu}^{-k}
\le 
\pq_j^2j!^2\, 
\xpar{\gl \nu}^{-j}, \qquad j,k\in\cS.
\end{equation}
Interchanging $j$ and $k$ we obtain equality in \eqref{a11a}.
Writing $\CC^2$ for the common value, \CCdef\CCa
and $\mu\=\sqrt{\gl\nu}$, we deduce that
\begin{equation}
  \label{a10a}
\pq_j=\CCa\frac{\mu^j}{j!},
\qquad j\in\cS.
\end{equation}

If, instead, $\nu=0$, then 
$\sum_{i\ge1} \nz_i\le 2\Nz=\o(n)$, so $\nz_0/n\to1$ and $\nz_i/n\to0$,
$i\ge1$; hence $\pq_0=1$, and $\pq_j=0$ for $j>0$. (Thus, $\nu=0$ implies
that $0\in\cS$.) Equation \eqref{a10a} holds in this case
also, this time with $\mu=0$. 

Hence, \eqref{a10a} holds in all cases. Summing over $j$ and
recalling \eqref{phis},  we find that
$1=\CCa\phis(\mu)$ and thus $\CCa=\phis(\mu)\qw$.
In particular, $\phis(\mu)>0$, another demonstration that
$\mu=0$ is possible only when $0\in\cS$.

In summary, along the selected subsequence,
\begin{equation}
  \label{a10b}
\pz_j\=
\frac{\nz_j}{n}\to\pq_j=\frac {\mu^j/j!}{\phis(\mu)}
=\PoS(\mu)\set{j},
\qquad j\in\cS,
\end{equation}
which is to say that every subsequence possesses a subsequence
along which
\begin{equation}
  \label{a11}
\ppz=(\pz_j)\ooo \to \PoS(\mu)
\end{equation}
for some $\mu$.
We next identify $\mu$, and show that it is the same for all subsequences.

We constructed $\nny$ above by changing by 2 the degrees of two vertices;
the reason was that this ensures that $\sum_i i \ny_i$ remains even.
If $j$ and $k$ have the same parity, \ie{} $j-k\equiv 0\pmod2$, then
we may also argue as above changing just one vertex degree from $k$ to
$j$. If further $k\ge j$, this leads as in \eqref{a9b} 
to the inequality
\begin{equation}
  \label{a12}
\nz_k k! \le (\nz_j+1)j!\,
\parfrac{2\gl \Nz}{n}^{(k-j)/2}.
\end{equation}
We consider again a subsequence along which \eqref{a11} holds for
some $\mu$. For every $k$ we apply \eqref{a12} with
$j$ the smallest number in $\cS$ of the same parity as $k$.
Using \eqref{a10b} for these (at most two)
$j$, we obtain, with
$\mu_n\=\xpar{2\gl \Nz/n}^{1/2}\to\xpar{\gl\nu}\qq=\mu$,
uniformly for all $k\in\cS$, 
\begin{equation*}
  \begin{split}
\nz_k
&\le
\frac1{k!} n\bigpar{\pq_j +\o(1)}j!\,\mu_n^{k-j} 
=
\frac n{\phis(\mu)k!}\mu_n^{k-j} \bigpar{\mu^j +\o(1)}
\\&
\le	
\frac n{\phis(\mu)k!}(\mu+1)^k \bigpar{1+\o(1)},
  \end{split}
\end{equation*}
since $\mu_n<\mu+1$ for large $n$.
Consequently, for every exponent $r>0$,
\begin{equation*}
  \sum_{k\in\cS} k^r \pz_k
=
  \sum_{k\in\cS} k^r \frac{\nz_k}n
\le
\bigpar{1+\o(1)}
\phis(\mu)\qw \sum_{k\in\cS}\frac{k^r}{k!}(\mu+1)^k=\O(1).
\end{equation*}
In other words, for every $r\in(0,\infty)$,
the distributions $\ppz$ have $r$th
moments that are uniformly bounded in $n$. 
It follows that all moments converge in \eqref{a11},
\ie, for every $r>0$,
\begin{equation}
  \label{a12a}
\sumks k^r\pz_k=
\sumks k^r\frac {\nz_k}n
\to
\sumks k^r\pq_k
=
\sumks k^r\PoS(\mu)\set{k}.
\end{equation}
In particular, $r=1$ yields, using \eqref{epos}
\begin{equation}
  \label{a12b}
\frac{2\Nz}{n}
=
\sumks k \frac{\nz_k}n
\to
\sumks k\pq_k
=\frac{\mu\phis'(\mu)}{\phis(\mu)}.
\end{equation}
On the other hand, we have assumed $2\Nz/n\to\nu$ and
$\mu=\sqrt{\gl\nu}$, whence $\nu=\mu^2/\gl$, so we have the
consistency relation
\begin{equation}
  \label{a13}
\frac{\mu\phis'(\mu)}{\phis(\mu)}=\frac{\mu^ 2}{\gl}.
\end{equation}
In other words, with $\egl$
defined by \eqref{egl}, we have $\mu\in\egl$. 

We summarize the result so far.
\emph{Each subsequence of $n$ possesses a subsequence such that
  \eqref{a11} and \eqref{a12a} hold for some $\mu\in\egl$, \ie,
$\mu\ge0$ satisfies \eqref{a13} and,
further, $\mu=0$ only if $0\in\cS$}.
 
The next step is to find the right solution of \eqref{a13} in the case
when $\egl$ contains two or more points. 

We continue to consider a subsequence for which \eqref{a11} holds.
By applying again \eqref{a12} with $j$ the smallest
odd or even number in $\cS$ as appropriate, 
and using $\nz_j\le n$ and \eqref{a9c}, 
we see that for some constants $\CC\CCdef{\Ce},\CC\CCdef{\Cm}$,
\begin{equation}
  \label{a14}
\nz_k k! \le \Ce n \Cm^k.
\end{equation}
If $k\ge\log n$, then by Stirling's formula, for large $n$,
\begin{equation*}
\log\bigpar{k! \,\Cm^{-k}}
\ge k\log k - k(1+\log\Cm) \ge 2k >\log(\Ce n),
\end{equation*}
and thus \eqref{a14} yields $\nz_k<1$. Consequently, for large $n$,
\begin{equation}
  \label{a14b}
\nz_k=0 \text{ for all } k\ge \log n.
\end{equation}

Let us now estimate $\annz=\max_{\nn}\ann$. 
By \eqref{a8} and
Stirling's formula, recalling \eqref{a5s},
\begin{multline*}
  \log \annz = n\log n - n + \O(\log n) 
- \sum_i\bigpar{\nz_i\log \nz_i - \nz_i + \O(\log(\nz_i+1))}
\\
- \sum_i\nz_i\log (i!)
+ \Nz\bigpar{\log(2\Nz)-1}+\O(1)+\Nz\log(\gl/n).
\end{multline*}
By \eqref{a14b}, we only have to sum over $i\le\log n$, and thus the
sum of all $\O$ terms is $\O(\log^2 n)$. Thus, with $y\=\sum_i i \pz_i=2\Nz/n$,
\begin{equation}\label{a15a}
  \begin{split}
\frac1n \log \annz  
&= \log n - 1 
- \sum_i\pz_i\bigpar{\log n+\log \pz_i-1+\log i!}
\\&\hskip8em
+ \frac y2\bigpar{\log y+\log\gl-1}+\o(1)
\\
&=
- \sum_i\pz_i{\log (\pz_i i!)}
+ \frac y2\bigpar{\log y+\log\gl-1}+\o(1)
.
  \end{split}
\end{equation}
For each $i$, $\pz_i\to\pq_i$ by \eqref{a11}, and in addition, 
by \eqref{a12b} and \eqref{a13},
\begin{equation}\label{a15x}
  y
=
\sum_i i \pz_i
\to
\sum_i i \pq_i
=
\mu^2/\gl.
\end{equation}
Now, $x\log x\ge - e\qw$ on $\ooox$, whence
$\pz_i\log (\pz_i i!)\ge -e\qw/i!$, and by \eqref{a14},
\begin{equation*}
  \pz_i\log (\pz_i i!)
\le \pz_i\log i!
\le \pz_i i\log i
\le \pz_i i^2
=\O\bigpar{i^2\Cm^i/i!}.
\end{equation*}
Consequently, by dominated convergence,
\begin{equation}\label{a16}
\frac1n \log \annz  
\to 
- \sum_i\pq_i{\log (\pq_i i!)}
+ \frac{\mu^2}{2\gl}\Bigpar{\log \frac{\mu^2}{\gl}+\log\gl-1}
.
\end{equation}
Furthermore, by 
\eqref{a10b} and 
\eqref{a15x},
if $\mu>0$, 
\begin{equation}\label{a16b}
\sum_i\pq_i{\log (\pq_i i!)}  
=
\sum_i\pq_i\bigpar{i\log\mu-\log\phis (\mu)}  
=
\frac{\mu^2}{\gl}\log\mu-\log\phis (\mu);  
\end{equation}
 if $\mu=0$ this holds trivially with all terms zero.
Hence, \eqref{a16} yields, using \eqref{a13},
\begin{equation}\label{a16a}
  \begin{split}
\frac1n \log \annz  
&\to 
-\frac{\mu^2}{\gl}\log\mu+\log\phis (\mu)
+ \frac{\mu^2}{2\gl}\bigpar{2\log \mu-1}
\\&
=
\log\phis (\mu)
- \frac{\mu^2}{2\gl}
=
\log\phis (\mu)
- \frac{\mu\phis'(\mu)}{2\phis(\mu)}
.	
  \end{split}
\end{equation}

Conversely, take any finite sequence $(p_i)_{i=0}^M$ with $p_i\ge0$,
$\sum_i p_i=1$ and $p_i=0$ when $i\notin\cS$.
Define $n_i=nx_i$, rounded up or down to integers, preserving $\sum_i
n_i = n$ and possibly adjusting two of them by $\pm1$ so that $\sum_i in_i$ is
even.
As \ntoo, we then obtain as in \eqref{a15a}--\eqref{a16}
(but simpler, since the sums are finite), with $\nu\=\sum_i ip_i$,
\begin{equation*}
\frac1n \log \ann  
\to 
- \sum_ip_i{\log (p_i i!)}
+ \frac{\nu}{2}\Bigpar{\log \nu+\log\gl-1}
.
\end{equation*}
Since, by definition, $\annz$ is maximal, $\annz\ge\ann$, and thus
\begin{equation}\label{a17}
\liminf_\ntoo\frac{\log \annz}n 
\ge
- \sum_i p_i{\log (p_i i!)}
+ \frac{\nu}{2}\Bigpar{\log \nu+\log\gl-1}
.
\end{equation}
We have shown \eqref{a17} for any probability distribution $(p_i)$ on
$\cS$ with finite support.
More generally, let  $(p_i)$ be a
probability distribution supported on $\cS$ with 
$\nu\=\sum_i ip_i<\infty$ and
$\sum_i p_i{\log (p_i i!)}<\infty$.
For $M\ge\min\cS$, 
let $p_i^{(M)}\=p_i\big/\!\sum_{j\le M}p_j$ for $i\le M$ and apply
\eqref{a17} to $(p_i^{(M)})_{i=0}^M$.
It is easily seen that the \rhs{} of \eqref{a17} converges as $M\to\infty$ 
to the corresponding value for $(p_i)$,
showing
that \eqref{a17} holds for $(p_i)$ also.

In particular, for any $\mu\in\egl$,
we can use \eqref{a17} with 
$p_i=\PoS(\mu)\set{i}$ given
by \eqref{pos} and, by \eqref{epos},
\begin{equation*}
 \nu\=
\sum_i i\PoS(\mu)\set{i}
=\frac{\mu\phis'(\mu)}{\phis(\mu)}
=\frac{\mu^ 2}{\gl}.
\end{equation*}
Hence, \eqref{a17} yields,
by the calculations in \eqref{a16b} and \eqref{a16a},
\begin{equation}\label{allahelgon}
  \begin{split}
\liminf_\ntoo\frac{\log \annz}n 
&\ge
\log\phis (\mu)
- \frac{\mu\phis'(\mu)}{2\phis(\mu)}
=\psis(\mu)
,	
  \end{split}
\end{equation}
for every $\mu\in\egl$.
Comparing this to \eqref{a16a}, we see that if \eqref{a11} holds for
some subsequence and some $\mu\in\egl$, then this $\mu$ must maximize
$  \psis(\mu)=\psisx(\mu;\gl)$ 
over $\egl$, in other words, $\mu=\mux$ as defined in \refT{T1}.
In particular, this shows that every subsequence possesses a subsequence
such that \eqref{a11} holds with a fixed $\mu=\mux$; hence
\eqref{a11} holds
for the full sequence of \ntoo, and
\begin{equation}
  \label{a11x}
\ppz=(\pz_j)\ooo \to \PoS(\mux).
\end{equation}

\begin{remark}\label{Rpsi2} 
  We have for simplicity considered $\mu\in\egl$ only in \eqref{allahelgon};
  for general $\mu\ge0$ the lower bound obtained from this argument 
takes the form,
with $\nu=\mu\phis'(\mu)/\phis(\mu)$,
\begin{equation*}
  \begin{split}
\log\phis (\mu)
+\frac{\nu}2\Bigpar{\log\frac{\nu\gl}{\mu^2}-1}
=
\psisxx(\mu;\gl)
  \end{split}
\end{equation*}
defined and studied in
\refS{Smu}. 
\end{remark}

We have so far studied the mode $\ppz$ of the degree distribution. We
now show that the distribution is concentrated close to the mode.

\begin{lemma}
  \label{L3}
For every $\eps>0$, there exists $\cc=\ccx(\eps)>0\ccdef{\ccliii}$ such
that, if $n$ is large 
enough then for every $\nn\in\cNSn$ with 
$\dtv\bigpar{\nn/n,\PoS(\mux)}\ge \eps$, 
$\ann\le e^{-\ccx n}\annz$.
\end{lemma}

We will first show a weaker statement.

\begin{lemma}
  \label{L3b}
For every $\eps>0$, there exists $\cc=\ccx(\eps)>0\ccdef{\ccliiib}$
such that, if 
$n$ is large 
enough then for every $\nn\in\cNSn$ with 
$\dtv\bigpar{\nn/n,\PoS(\mux)}\ge \eps$, either
$\ann\le e^{-\ccx n}\annz$, or there exists $\nny\in\cNSn$ with 
$\dtv(\nn/n,\nny/n)\le 2/n$ and $\ann\le(1-\ccx)\anny$.
\end{lemma}

\begin{proof}
Suppose this fails. Then there exists $\eps>0$ and a sequence
$\nn=\nn\wwx n\in\cNSn$  with 
$n\to\infty$, such that
$\dtv\bigpar{\nn/n,\PoS(\mux)}\ge \eps$, 
$\ann = e^{-\o(n)}\annz$ and $\anny\le(1+\o(1))\ann$
for all $\nny\in\cNSn$ with 
$\dtv(\nn/n,\nny/n)\le 2/n$,
\ie{}, for all $\nny\in\cNSn$ with $\sum_i|n_i-\ny_i|\le 4$.

We now repeat much of the arguments presented above for the mode $\nnz$.
First, we obtain that \eqref{a9}, \eqref{a9b} and \eqref{a12} hold for
these $\nn$, with an extra factor $1+\o(1)$ on the \rhs{s}, uniformly
in all $j,k\in\cS$ (with $k\ge j$ for
\eqref{a9b} and $k\ge j$ and $k\equiv j\pmod2$ for \eqref{a12}).
Furthermore, by \refL{L2} and the assumption  
$\ann = e^{-\o(n)}\annz$, 
$N\=\half\sum_i i n_i \le Bn$ (for large $n$).
It follows, as above, that by considering a subsequence we may assume
that $2N/n\to\nu\in[0,\infty)$ and $\nn/n\to\ppy$ for some probability
  distribution $\ppy$, where, again as above, necessarily
$\ppy=\PoS(\mu)$ for some $\mu\in\egl$ and \eqref{a16a} holds for
  $\nn$.

Since $\log\ann=\o(n) + \log\annz$, this shows that
$\psis(\mu)=\psis(\mux)$, and thus $\mu=\mux$, since $\mux$ is
assumed to be a unique maximum point.
Consequently, $\nn/n\pto\PoS(\mux)$, which contradicts
$\dtv\bigpar{\nn/n,\PoS(\mux)}\ge \eps$.
\end{proof}

\begin{proof}[Proof of \refL{L3}]
By \refL{L3b}, if $n$ is large
  enough,
$\dtv\bigpar{\nn/n,\PoS(\mux)}\ge \eps$, and
$\ann> e^{-\ccliiib n}\annz$,
there exists $\nn\wwi=\nn'$ such that
$\dtv(\nn/n,\nn\wwi/n)\le 2/n$ and $\ann\le(1-\ccx)z(\nn\wwi)$; in
  particular, $z(\nn\wwi)>\ann$.

If also $\dtv\bigpar{\nn\wwx1/n,\PoS(\mux)}\ge \eps$, we iterate and
  find
$\nn\wwx2$, and so on.
This gives a sequence $\nn\wwx0=\nn, \nn\wwx1, \dots,\nn\wwx L$,
  where for $l<L$ we have 
$\dtv(\nn\wwx l/n,\nn\wwx{l+1}/n)\le 2/n$ and
  $z(\nn\wwx{l})\le(1-\ccx) z(\nn\wwx{l+1})$, while
$$
\dtv\bigpar{\nn\wwx L/n,\PoS(\mux)}< \eps.
$$

If further
$\dtv\bigpar{\nn/n,\PoS(\mux)}\ge 2\eps$, it follows that 
$\dtv\bigpar{\nn/n,\nn\wwx{L}/n}> \eps$, and thus the number of steps
$L>\eps n/2$. Consequently,
\begin{equation*}
  \ann
\le(1-\ccx)^L z\bigpar{\nn\wwx L}
\le(1-\ccx)^L \annz
\le\exp\bigpar{-\tfrac12\ccx\eps n} \annz.
\end{equation*}
This proves \refL{L3} for $2\eps$, with
$\ccliii(2\eps)=\min\{1,\tfrac12\eps\}\ccliiib(\eps)$.
\end{proof}

We now complete the proof of \refT{T1x}.
Let $\eps\ge0$ and let, with $B$ as in \refL{L2}
and $N=\half\sum_i in_i$,
\begin{align*}
 A_1&\=\set{\nn\in\cNSn:N>Bn}, 
\\
\aae&\=
\set{\nn\in\cNSn:N\le Bn \text{ and }\dtv\bigpar{\nn/n,\PoS(\mux)}\ge\eps}.
\end{align*}

\begin{lemma}
  \label{LA}For every $\eps\ge0$,
$|\aae|\le |\aao|
=e^{\o(n)}$.
\end{lemma}

\begin{proof}
  Suppose $\nn=(n_i)_i\in\aao$. For each $i$, $n_i\le n$, and thus the
  number of choices of $(n_i)_{i\le\sqrt n}$ is at most 
$(n+1)^{\sqrt n+1}=\exp\bigpar{\O(\sqrt n\log n)}$.
Furthermore, $\sum_i i n_i = 2N \le 2Bn$, and thus $n_i=0$ for $i>2Bn$
  and
  \begin{equation*}
\sum_{i>\sqrt n} n_i \le \frac{2Bn}{\sqrt n}=2B\sqrt n,	
  \end{equation*}
so $(n_i)_{\sqrt n<i\le 2BN}$ may be described by a sequence of at
most $2B\sqrt n$ numbers in the range $[\sqrt n,2Bn]$ (the degrees of
the corresponding vertices). Hence, the number of choices of 
$(n_i)_{i>\sqrt n}$ is at most 
$(2Bn)^{2B\sqrt n}=\exp\bigpar{\O(\sqrt n\log n)}$. 

Combining the two parts, 
$|\aao|=\exp\bigpar{\O(\sqrt n\log n)}$. 
\end{proof}

Now, fix $\eps>0$. By Lemmas \ref{L2}, \refand{L3}{LA},
$$
\P\lrpar{\dtv\bigpar{\pix(\gnlns),\PoS(\mux)}\ge\eps}
\le e^{-n} + \frac{|\aae| e^{-\ccliii n}\annz}{\zxnls}
\le e^{-\cctic n},
$$ 
for some $\cctic>0$ and all large $n$.
This proves \eqref{t1c} and hence \eqref{t1a}
(for $\gxnls$).

A similar calculation with $\eps=0$ yields
\begin{equation*}
\annz\le\zxnls =
\sum_{\nn\in A_1}\ann 
+ \sum_{\nn\in \aao}\ann 
\le  e^{\o(n)}\annz,
\end{equation*}
and thus $\log\zxnls=\log\annz+\o(n)$,
which together with \eqref{a16a} implies
$n^{-1}\log\zxnls\to\psis(\mux)$.
By \eqref{b1},
this further yields \eqref{zx}. 

\begin{lemma}
  \label{LN}
Uniformly in all $k\ge0$,
  \begin{equation*}
\E n_k(\gxnls) 
\le 
\CC n e^{-\cc k}.
  \end{equation*}\CCdef\CCLN \ccdef\ccLN
Moreover, for any $\gl_0$, this holds uniformly in $\gl\le\gl_0$.
\end{lemma}

\begin{proof}
Let $j$ be the smallest element of $\cS$ with the same parity as $k$.
Given any $\nn\in\cNSn$, let $\nny\in\cNSn$ be given by
$\ny_j\=n_j+1$, $\ny_k\=n_k-1$ and $\ny_i\=n_i$, $i\neq j,k$
(assuming $j<k$ and $n_k\ge1$; otherwise $\nny=\nn$).
By \eqref{a8}, \cf{} the argument yielding \eqref{a12},
\begin{equation*}
\ann\le\anny \frac{(n_j+1)j!}{n_k k!}
\parfrac{2\gl N}{n}^{(k-j)/2},
\end{equation*}
and thus
\begin{equation}\label{c1}
n_k\ann\le \frac{\CC n}{k!}
\parfrac{2\gl N}{n}^{(k-j)/2}
\anny.
\end{equation}
\refL{L2} and \eqref{c1} 
imply
\begin{equation*}
\begin{split}
    \sum_{\nn\in\cNSn}n_k\ann
&=
  \sum_{\nn\in A_1}n_k\ann
+  \sum_{\nn\in\aao,\;n_k>0}n_k\ann
\\&
\le n e^{-n} \annz+
\frac{\CCx n}{k!}
\xpar{2B\gl}^{(k-j)/2}
\sum_{\nny\in\cNSn}\anny
\\&
\le n e^{-n} \zxnls+
\CC\frac{\CC^k n}{k!}
\zxnls
\end{split}
\end{equation*}
and the result for $0\le k\le 2Bn$ follows by dividing by $\zxnls$,
with $\ccLN=1/(2B)$.

Finally, if $k>2B n$, then for every $i\ge1$ we have
$n_k \ge i \implies N\ge k n_k/2 \ge ki/2 >  Bn$, and thus by \refL{L2}
\begin{equation*}
  \P\bigpar{n_k(\gxnls)\ge i}
\le
  \P\bigpar{\edg(\gxnls)\ge ki/2}
\le e^{-ki/(2B)}.
\end{equation*}
Hence,
$
\E n_k(\gxnls)
=
\sumii \P\bigpar{n_k(\gxnls)\ge i}
\le 2 e^{-k/(2B)}
$.
\end{proof}

Let $X_{n,K}\=n^{-1}\sum_{k=0}^K k^r n_k(\gxnls)$
be a partial sum of the sum in \eqref{t1b} for $\gxnls$. Then,
for every fixed $K$, by \refL{LN},
\begin{equation*}
  \begin{split}
\E(X_{n,\infty}-X_{n,K})
&=
\E \lrpar{\sum_{k=K+1}^\infty \frac{k^r  n_k(\gxnls)}{n}}
\le
\sum_{k=K+1}^\infty \CCLN k^r e^{-\ccLN k},	
  \end{split}
\end{equation*}
which can be made arbitrarily small by choosing $K$ large.
Since 
$X_{n,K}\pto\sum_{k=0}^K k^r \PoS(\mux)\set k$ as \ntoo{} for every
fixed $K$, \eqref{t1b} follows by standard arguments.
(See, for example, the much more general
\cite[Theorem 4.2]{Bill}.) This completes the proof of \refT{T1x}.

\section{Proof of \refL{L1xy} and \refT{T1}}\label{SpfT1}

\begin{proof}[Proof of \refL{L1xy}]  
We use \refL{Lconf} together with the result of \cite{SJ195} (with
previous partial results by many authors) that states that, for a sequence of
degree sequences $\dd=\dd^{(n)}$ satisfying $\sum_id_i\to\infty$,
if $\sum_i d_i^2 = \O\bigpar{\sum_i d_i}$,
 then
$\liminf \P\bigpar{\gnddx\simple}>0$.
(The converse holds also, see \cite{SJ195}.)
In other words,
for every $K$ there exist constants $a_K$ and $b_K>0$ such that, if 
\begin{align}\label{sj195}
\text{(i)}\q \sum_id_i\ge a_K 
\qquad
\text{and} 
\qquad
\text{(ii)}
\q
\sum_i d_i^2 \le K \sum_i d_i, 
\end{align}
then
\begin{equation}
  \P\bigpar{\gnddx\simple}\ge b_K.
\end{equation}
Let $p(\dd)\=\P\bigpar{\gnddx\simple}$. By \refL{Lconf}, for every $K$, 
\begin{equation}\label{c0}
  \begin{split}
  \P\bigpar{\gxnlns\simple}
&=\E p\bigpar{\dd(\gxnlns)}
\\&
\ge b_K\P\bigpar{\dd(\gxnlns)\text{ satisfies (\ref{sj195})}}.	
  \end{split}
\end{equation}
Thus, it suffices to show that 
$\liminf\P\bigpar{\dd(\gxnlns)\text{ satisfies \eqref{sj195}}}>0$.

First, consider the case $\mux>0$. By \refT{T1x},
$$
\frac1n\sum_id_i(\gxnlns)^r\pto A_r,
\qquad 
r=1,2,
$$
for some constants $A_r>0$.
Hence, taking $K\=A_2/A_1+1$, 
$$\P\bigpar{\dd(\gxnlns)\text{ satisfies \eqref{sj195}}}\to1$$
and the result follows in this case. 

Now suppose that $\mux=0$, which can occur only if
$0 \in \cS$.
 Although the graphs are sparser in this
case, and intuitively it seems more probable that they are simple, we
have not found a really simple proof and have to work harder in this
case. (The proof above is not valid since now $A_1=A_2=0$.)
Let $\cS'\=\cS\setminus\set0$, and let $M\=n-n_0$ be the number of
non-isolated vertices in $\gxnlns$.

Let $0\le m\le n$ and let $V$ be any subset of \set{1,2,\dots,n} with
$|V|=m$.
If we condition $\gxnlns$ on having the set of non-isolated vertices
equal to $V$, we evidently get a random multigraph $\gxmn$ on $V$ (up
to relabelling the vertices) with $n-m$ isolated vertices added. It
follows that, for $r>0$,
\begin{equation}
  \label{kia}
\Bigpar{\sum_id_i(\gxnlns)^r\Bigmid M=m}
\eqd
\sum_id_i(\gxmn)^r.
\end{equation}
Note that the relevant parameter of $\gxmn$ is $m\gl_n/n$, not
$\gl_n$.
Since we consider $0\le m\le n$, and the case $m=0$ is trivial and
thus can be ignored, we have
$0<m\gl_n/n\le \gl_n\le\CC$, and thus \refL{LN} implies that
\begin{equation*}
  \E\sum_id_i(\gxmn)^2
=
\E\sumki k^2 n_k\bigpar{\gxmn}
\le \CC m,
\qquad m\le n,
\end{equation*}
for some constant $\CCx$ not depending on $m$ or $n$.
Furthermore, since $0\notin\cS'$, each vertex degree is at least 1 and 
$\sum_id_i(\gxmn)\ge m$. Consequently, choosing $K=4\CCx$, it
follows by Markov's inequality 
that, for every $m\le n$,  with probability at least $\frac34$,
\begin{equation*}
\sum_id_i(\gxmn)^2
\le 4\CCx m
\le K \sum_id_i(\gxmn).
\end{equation*}
Consequently, by conditioning on $M$ and using \eqref{kia}, 
\begin{equation*}
 \P\bigpar{\text{\eqref{sj195}(ii) holds for }\gxnlns} 
\ge \tfrac34.
\end{equation*}

Hence, whenever
\begin{equation}\label{supp}
\P\biggpar{\sum_id_i(\gxnlns)\ge a_K}\ge \tfrac12,
\end{equation}
 then
\eqref{sj195} fails for $\gxnlns$
with probability at most $\frac12+\frac14=\frac34$,
and thus by \eqref{c0}
\begin{equation*}
  \P\bigpar{\gxnlns\simple}\ge\tfrac14b_K.
\end{equation*}

The only remaining case is when $\mux=0$ and
\eqref{supp} is false.
We recall $\sum_i d_i(G)=\sum_j j n_j(G)$ and define
$\cnx\=\set{\nn\in\cNSn:\sum_j j n_j <a_K}$.
Thus we now have, using \eqref{a8c},
\begin{equation}
  \label{c3a}
\tfrac12 < \P\bigpar{\nn(\gxnlns)\in\cnx}
=\frac1{\zxnlns}\sum_{\nn\in\cnx} \ann .
\end{equation}
Further, $\cnx$ is a finite set 
(with $n_j=0$ for $j\ge a_K$), and \eqref{a8} yields
\begin{equation}
  \label{c3}
\ann
\le
\frac{n!}{n_0!}(\ceil{a_K})!!\,\parfrac{\gl_n}{n}^{\frac12\sum j n_j}
\le
\CC
n^{n-n_0}
\parfrac{\gl_n}{n}^{\frac12\sum j n_j},
\quad
\nn\in\cnx.
\end{equation}
We now use \refT{Tmux0}, which shows that $\mux=0$ is possible only
when $1\notin\cS$.
Thus, if $\nn\in\cnx\subset\cNS$, then $n_1=0$ and
$\frac12\sum_j j n_j \ge \sum_2^\infty n_j = n-n_0$;
hence, since $\gl_n=\O(1)=\O(n)$,
by \eqref{c3},
\begin{equation*}
\ann
\le
\CC
n^{n-n_0}
\parfrac{\gl_n}{n}^{n-n_0}
\le\CC,
\qquad
\nn\in\cnx.
\end{equation*}
Therefore, by \eqref{c3a},
\begin{equation*}
  \zxnlns \le 2 \sum_{\nn\in\cnx} \ann
\le\CC.
\end{equation*}
However, if $E_n$ is the empty graph with $n$ vertices and no edges,
then by \eqref{b0},
$\P\bigpar{\gxnlns=E_n}=1/\zxnlns\ge \CCx\qw$. Now $E_n$ is
simple, and so the result follows in this case too.
\end{proof}

\begin{proof}[Proof of \refT{T1}]
  By Lemmas \refand{Lsimple}{L1xy}, for any event $\cE$ and $n$ large
  enough,
  \begin{equation*}
\P\bigpar{\gnlns\in\cE}
\le
\frac{\P(\gxnlns\in\cE)}{\P(\gxnlns\simple)}
\le
\CC\P(\gxnlns\in\cE).
  \end{equation*}
Hence parts (i)--(iii) follow directly from \refT{T1x}.
Similarly,
  \begin{equation*}
	\begin{split}
\P&\bigpar{\gnln\text{ is an \sgraph}}
=
\P\bigpar{\gxnln\text{ is an \sgraph}\mid\gxnln\simple}
\\&
=
\P\bigpar{\gxnln\simple\mid\gxnln\text{ is an \sgraph}}
\frac{\P\bigpar{\gxnln\text{ is an \sgraph}}}{\P\bigpar{\gxnln\simple}}
\\&
=
\frac{\P\bigpar{\gxnlns\simple}}{\P\bigpar{\gxnln\simple}}
\P\bigpar{\gxnln\text{ is an \sgraph}}
	\end{split}
  \end{equation*}
and \ref{T1d} follows by \refT{T1x} and \refL{L1xy} (applied both to
$\cS$ and with $\cS$ replaced by $\bbZo$). 
\end{proof}

\section{Proofs of Theorems \ref{Tgiant} and \ref{Tcore}}\label{sec4.5}

\begin{proof}[Proof of \refT{Tgiant}]
The case $\mux=0$ is trivial by \refT{T1}, as remarked in
  \refR{Rgiant0}, so we will assume $\mux>0$.
  We use the 
results of \citet{MR95,MR98} in the following version, see 
\citet[Theorem 2.3 and Remark 2.7]{SJ204}; we only consider the
limiting degree distribution 
$(p_k)_{k=0}^\infty$ given by $p_k=\PoS(\mux)\set k$.

Let $\nu=\nu(\mux)$, $Q=Q(\mux)$, $\xix$, $\gammax$ and $\zetax$
be as in \refS{giantcore}; the existence of a unique solution
$\xix\in(0,1)$ in (i) follows by \cite[Lemma 5.5]{SJ204}.
By assumption, $p_0+p_2<1$. 
Further, let $\gnd$ be the random graph with given degree
sequence $\dd$, chosen uniformly among all such graphs (assuming that
there is at least one), and let $\Gnd$ and $\GGnd$ be the
largest and second largest components of $\gnd$.

\begin{theorem}
  \label{TMR}
Suppose that, for each $n$,  
$\dd=(d_i)_1^n$ is a sequence of non-negative
integers such that $\sumin d_i$ is even, and that
\begin{romenumerate}
  \item\label{C1p}
$
|\set{i:d_i=k}|/n\to p_k$ as \ntoo, for every $k\ge0$;
\item \label{C1d2}
$\sumin d_i^ 2=\O(n)$.
\end{romenumerate}
Then, the following hold for the random  graph \gnd, as \ntoo:
\begin{align*}
  \ver(\Gnd)/n&\pto \gammax,
&
  \edg(\Gnd)/n&\pto \zetax,
\\
  \ver(\GGnd)/n&\pto 0,
&
  \edg(\GGnd)/n&\pto 0.
\end{align*}
\end{theorem}

This theorem is stated as a limit result, but it can be reformulated
as follows.

\begin{theorem}
  \label{TMRfinit}
For every $\eps>0$ and $C<\infty$, there exists $\gd>0$ such that if
$n>\gd\qw$ and $\dd=\dn$ is a degree sequence 
such that $\sumin d_i$ is even and 
\begin{romenumerate}
  \item\label{C1pf}
$\sumk\big||\set{i:d_i=k}|/n- p_k\big|<\gd$,
\item \label{C1d2f}
$\sumin d_i^2\le C n$,
\end{romenumerate}
then
\begin{align*}
  \P&\bigpar{|\ver(\Gnd)/n- \gammax| > \eps}<\eps,
&
  \P&\bigpar{|\edg(\Gnd)/n- \zetax| > \eps}<\eps,
\\
  \P&\bigpar{\ver(\GGnd)/n > \eps}<\eps,
&
  \P&\bigpar{\edg(\GGnd)/n > \eps}<\eps.
\end{align*}
\end{theorem}

By \refT{T1}, for every $\eps>0$, a suitable $C$ and sufficiently
large $n$, the random degree sequence 
$\dd(\gnlns)$ satisfies the conditions (i) and (ii) of \refT{TMRfinit}
with probability at least $1-\eps$. 

Since
$(\gnlns\mid\dd(\gnlns=\dd)\eqd\gnd$ by Lemmas \refand{Lconf}{Lsimple}, it
follows that 
$  \P\bigpar{|\ver(\Gnlns)/n- \gammax| > \eps}<2\eps$ if $n$ is large
enough, and similarly for $\edg(\Gnlns)$ and for $\GGnlns$, which
proves \refT{Tgiant}.
\end{proof}

\begin{proof}[Proof of \refT{Tcore}]
 This proof is similar, using \cite[Theorem 2.4]{JL}; we omit the
 details while noting that we now need condition (i) of
 \refT{TMRfinit}, and in addition the condition (ii$'$)
(stronger than (ii) above)
that  $\sumin e^{c d_i} \le Cn$ for some $c>0$.
 This holds for $\dd(\gnlns)$ with
 probability $>1-\eps$ for suitable $c$ and $C$ (that may depend on
 $\eps$), as a consequence of the following corollary of \refL{LN}.

\begin{lemma}
  \label{Lexp}
Assume that $\gl_n\to\gl>0$.
If $c<\ccLN$, then
  \begin{equation*}
\E \sumin e^{c d_i(\gxnlns)} = \E \sumk n_k(\gxnlns) e^{ck} \le \CC n.
  \end{equation*}
\end{lemma}
By Lemmas \ref{Lsimple} and \ref{L1xy}, the conclusion of the lemma
is valid for $\gnlns$ also.
\end{proof}

\section*{Acknowledgement}
This research was mainly done during a visit by SJ to the University
of Cambridge, partly funded by Trinity College.

\newcommand\AAP{\emph{Adv. Appl. Probab.} }
\newcommand\JAP{\emph{J. Appl. Probab.} }
\newcommand\JAMS{\emph{J. \AMS} }
\newcommand\MAMS{\emph{Memoirs \AMS} }
\newcommand\PAMS{\emph{Proc. \AMS} }
\newcommand\TAMS{\emph{Trans. \AMS} }
\newcommand\AnnMS{\emph{Ann. Math. Statist.} }
\newcommand\AnnPr{\emph{Ann. Probab.} }
\newcommand\CPC{\emph{Combin. Probab. Comput.} }
\newcommand\JMAA{\emph{J. Math. Anal. Appl.} }
\newcommand\RSA{\emph{Random Struct. Alg.} }
\newcommand\ZW{\emph{Z. Wahrsch. Verw. Gebiete} }
\newcommand\DMTCS{\jour{Discr. Math. Theor. Comput. Sci.} }

\newcommand\AMS{Amer. Math. Soc.}
\newcommand\Springer{Springer-Verlag}
\newcommand\Wiley{Wiley}

\newcommand\vol{\textbf}
\newcommand\jour{\emph}
\newcommand\book{\emph}
\newcommand\inbook{\emph}
\def\no#1#2,{\unskip#2, no. #1,} 
\newcommand\toappear{\unskip, to appear}

\newcommand\webcite[1]{
   \penalty0\texttt{\def~{{\tiny$\sim$}}#1}\hfill}
\newcommand\webcitesvante{\webcite{http://www.math.uu.se/~svante/papers/}}
\newcommand\arxiv[1]{\webcite{http://arxiv.org/#1}}

\def\nobibitem#1\par{}

\end{document}